\documentclass[a4paper,draft]{amsart}
\usepackage{latexsym}
\usepackage{ifthen}\usepackage[leqno]{amsmath}
\usepackage{enumerate}
\usepackage{calc}
\usepackage{hyphenat}
\usepackage {amstext,amsbsy,amsopn,amsthm,amsgen,amsfonts,amscd,amsxtra,upref}
\usepackage{mathrsfs}
\usepackage{euscript}
\usepackage{amssymb}
\theoremstyle{plain}
\newtheorem{thm}{Theorem}[section]
\newtheorem{lem}[thm]{Lemma}
\newtheorem{cor}[thm]{Corollary}

\newtheorem{prp}[thm]{Properties}

\newtheorem{fact}[thm]{Fact}

\theoremstyle{definition}
\newtheorem{df}[thm]{Definition}
\newtheorem{exm}[thm]{Example}
\newtheorem{exs}[thm]{Examples}
\newtheorem{prob}[thm]{Problem}

\theoremstyle{remark}
\newtheorem{rem}[thm]{Remark}
\newtheorem{rms}[thm]{Remarks}

\newcommand{\CCC}{\mathbb{C}}

\newcommand{\NNN}{\mathbb{N}}
\newcommand{\RRR}{\mathbb{R}}
\newcommand{\SSS}{\mathbb{S}}\newcommand{\TTT}{\mathbb{T}}

\newcommand{\ZZZ}{\mathbb{Z}}%% '\mathbb k'

\newcommand{\FFf}{\CMcal{F}}
\newcommand{\HHh}{\CMcal{H}}\newcommand{\IIi}{\CMcal{I}}
\newcommand{\KKk}{\CMcal{K}}

\newcommand{\PPp}{\CMcal{P}}
\newcommand{\SSs}{\CMcal{S}}

\newcommand{\Aa}{\mathfrak{A}}\newcommand{\Bb}{\mathfrak{B}}\newcommand{\Cc}{\mathfrak{C}}
\newcommand{\Ff}{\mathfrak{F}}
\newcommand{\Ii}{\mathfrak{I}}
\newcommand{\Kk}{\mathfrak{K}}

\newcommand{\Uu}{\mathfrak{U}}

\newcommand{\Zz}{\mathfrak{Z}}

\usepackage{amsmath} 
\usepackage{amsfonts}
\usepackage{amssymb} 
\usepackage[T1]{fontenc} 
\usepackage{lmodern}   
\usepackage{ifthen}
\usepackage{hyphenat}

\numberwithin{equation}{section}
\newcommand{\eps}{\varepsilon}

\newcommand{\cl}{\overline}

\newcommand{\lin}{\textup{lin}}

\newcommand{\oo}{\infty}
\newcommand{\id}{\textup{id}}
\newcommand{\diam}{\textup{diam}}
\newcommand{\Rep}{\textup{Rep}}
\newcommand{\Irr}{\textup{Irr}}
\newcommand{\w}{\omega}

\newcommand{\Orb}{\textup{Orb}}
\newcommand{\Prob}{\textup{Prob}}

\newcommand{\Hom}{\textup{Hom}}

\newcommand{\Lip}{\textup{Lip}}
\title{Uniformity in $C^*$-algebras}
\author{Adam Wegert}
\address{A. Wegert{}\\Faculty of Applied Mathematics\\{}AGH University of Science and Technology\\{}al. Mickiewicza 30\\{}30-059 Krak\'{o}w\\{}Poland}
\email{a\_wegert@o2.pl}
\date{}

\begin{document}

%http://sjp.pwn.pl/slowniki/taki-%C5%BCe.html-czasem piszemy przecinek przed "taki ¿e"
%kropka po klamrach ???
%zarowno,..., jak i ...,otrzymujemy?

%namiary na ksiazki o grupach kwantowych

%http://tex.stackexchange.com/questions/57743/how-to-write-%C3%A4-and-other-umlauts-and-accented-letters-in-bibliography <--umlauty
%namiary na *produkty i deformacje i zwiazki z Hochschildem
%przed ZAŒ ma byæ przecinek
%przed ABY ma byæ przecinek
%przed CZY stawia sie przecinek np. w "`Powiedz mi, czy to prawda"' ale nie w "Kawa czy herbata"
%przecinek stawia siê przed LECZ
%przecinek stawia sie przed DLATEGO jak rowniez czasem po DLATEGO np.
%Musze sprzedaæ swoje stare wazony, dlatego wezmê je jutro na rynek
%Nie zda³a egzaminu dlatego, ¿e by³a nienauczona.
%WTEDY I TYLKO WTEDY, GDY-ma byc tam przecinek
%cCHYBA ¯E bez przecinka; przecinek przed CZYLI, ZATEM
%przecinek stawia siê przed A TAK¯E albo przed JAK ROWNIE¯
%przecinek stawia siê przed WIÊC
%ALBO: albo pojde do kina, albo do teatru/ pojde do kina albo do teatru
%przed LUB nie stawiamy przecinka
%przecinek stawia siê przed ALE

\begin{abstract} 
We introduce a notion of a uniform structure on the set of all representations of a given separable, not necessarilly commutative $C^*$-algebra $\Aa$ by introducing a suitable family of metrics on the set of representations of $\Aa$ and investigate its properties. We define the noncommutative analogue of the notion of the modulus of continuity of an element in $C^*$-algebra and we establish its basic properties. We also deal with morphisms of $C^*$-algebras by defining two notions of uniform continuity and show their equivalence.  
\end{abstract}

\maketitle

\section*{Introduction}

The famous theorem of Gelfand and Najmark establishes a one-to-one correspondence between 
compact topological spaces and unital, commutative $C^*$-algebras. This correspondence is well behaved in the sense that it is in fact a natural equivalence between categories of compact topological spaces and unital, commutative $C^*$-algebras. Therefore various topological properties of spaces may be translated into algebraic properties of the corresponding $C^*$-algebra. The philosophy of \textit{noncommutative} topology is to translate topological properties of spaces into the language of algebras and check whether the assumption of commutativity of a given algebra is necessary: if not, then one can state the definition in the context of noncommutative algebras and think that the underlying \textit{noncommutative space} possess given topological property. Our aim is to discuss the topological, or rather metric-space theoretic, notion of \textit{uniformity} in the context of noncommutative $C^*$-algebras.

\section{Notation and terminology}

In this section we will recall standard definitions just to fix notation. 
All considered vector spaces will be over the field $\CCC$ of complex numbers. $C^*$-algebras will be usually denoted by $\Aa,\Bb,\Cc$ etc. For an algebra $\Aa$, $\Zz(\Aa)$ will denote its center. When two objects $X,Y$ (in a given category) are isomorphic we will denote this fact by $X \cong Y$ (the category should be clear from the context). In the context of (unital) $C^*$-algebras, isomorphism is understood as (unital) bijective, $*$-preserving homomorphism (which is automatically isometric). For an element $x$ in a unital $C^*$-algebra by $r(x)$ we will denote its spectral radius defined as $r(x)=\sup\{|\lambda|: \lambda \in \sigma(x)\}$
where $\sigma(x)$ denotes \textit{the spectrum} of $x$. By $\SSs(\Aa)$ we will denote the state-space of $\Aa$ and $\PPp(\Aa)$ will denote its subspace consisting of pure states. $\HHh$ will stand for a (usually separable) Hilbert space, $\Bb(\HHh)$ for the algebra of all bounded operators on $\HHh$ and $\Uu(\HHh)$ for the group of unitary operators on $\HHh$.  
By $I$ we will mean the identity operator, occasionally we will write $I_{\HHh}$ to indicate in which space this operator acts. 
For $S \in \Bb(\HHh), T \in \Bb(\KKk)$ by $S \oplus T$ we will understand the operator in $\Bb(\HHh \oplus \KKk)$ defined by $(S \oplus T)(x,y):=(Sx,Ty)$. For a family of $C^*$-algebras $(\Aa_i)_{i \in \IIi}$ by $\prod_{i \in \IIi}\Aa_i$ we will denote the set$$\{(a_i)_{i \in \IIi}:a_i \in \Aa_i, \ i \in \IIi, \ \sup_{i \in \IIi}\|a_i\|<\oo\}.$$ With pointwise operations and supremum norm it becomes a $C^*$-algebra.
If all $\Aa_i$'s are unital, then so is $\prod_{i \in \IIi}\Aa_i$. On the other hand, by
$\bigoplus_{i \in \IIi}\Aa_i$ we will mean the \textit{direct sum} of algebras $\Aa_i$, i.e. the set of all $(a_i)_{i \in \IIi}$ \textit{vanishing at infinity} (meaning that for each $\eps>0$ there is a finite set $\IIi_{\eps} \subset \IIi$ such that for all $i \in \IIi \setminus \IIi_{\eps}$, $\|a_i\|<\eps$ holds) with pointwise operations. For a nonunital $C^*$-algebra $\Aa$ we denote by $\Aa^{+}$ the \textit{unitization} of $\Aa$. Everytime we have a short exact sequence of $C^*$-algebras of the form
\begin{equation} \label{exact} 0 \longrightarrow \Aa' \stackrel{\varphi}\longrightarrow \Aa \stackrel{\psi}\longrightarrow \Aa'' \to 0
\end{equation}
we would say that $\Aa$ is an \textit{extension} of $\Aa''$ by $\Aa'$. By a
\textit{representation} of a $C^*$-algebra (on a Hilbert space $\HHh$) we mean a $*$-homomorphism $\pi:\Aa \to \Bb(\HHh)$ and if $\Aa$ is unital we will usually assume that $\pi(1)=I$. The space on which the algebra $\Aa$ acts via representation $\pi$ will occasionally be denoted by $\HHh_{\pi}$.

Given two representations $\pi_i:\Aa \to \Bb(\HHh_i), \ i=1,2$ we will denote by $\pi_1 \oplus \pi_2$ their direct sum defined by the formula $\pi_1 \oplus \pi_2: \Aa \ni a \mapsto \pi_1(a) \oplus \pi_2(a) \in \Bb(\HHh_1 \oplus \HHh_2)$. Similarly we define the direct sum $\bigoplus_{i \in \IIi}\pi_i$ of an arbitrary family of representations $\{\pi_i\}_{i \in \IIi}$. In particular if for $i,j \in \IIi$ we have $\pi_i=\pi_j=:\pi$ and $|\IIi|=\alpha$, then we will denote the sum $\bigoplus_{i \in \IIi}\pi_i$ by $\alpha \odot \pi$. Occasionally we will use the same notation for operators in Hilbert spaces. 
On the set of all representations of $\Aa$ on a Hilbert space $\HHh$ we can define the \textit{point-norm} convergence as follows: if $\{\pi_{\sigma}\}_{\sigma}$ is a net of representations of $\Aa$ on a space $\HHh$ we declare $\pi_{\sigma} \to \pi$ in the point-norm topology, if for any $a \in \Aa$ we have $\pi_{\sigma}(a) \stackrel{\| \cdot \|}{\longrightarrow} \pi(a)$.
In other words, it is the topology of the family of mappings $\pi \mapsto \|\pi(a)\|$, $a \in \Aa$. We also consider the \textit{compact-open} topology, defined as follows: given a net  $\{\pi_{\sigma}\}_{\sigma}$ of representations we declare $\pi_{\sigma} \to \pi$ in the compact-open topology if for any compact set $L \subset \Aa$ we have $\sup_{a \in L}\|\pi_{\sigma}(a)-\pi(a)\| \to 0$. In other words, this is the topology of convergence on compacta. 
For any two unital  $C^*$-algebras $\Aa,\Bb$ by $\Hom(\Aa,\Bb)$ we will mean the set of all unital *-homomorphisms $\Aa \to \Bb$.

\section{Preliminaries}
A well known result of Gelfand and Najmark establishes a one-to-one correspondence between unital commutative $C^*$-algebras and compact (Hausdorff) topological spaces. In the case of nonunital algebras one have to deal with spaces which are \textit{locally} compact. This correspondence is in fact functorial and allows one to translate topological properties into the algebraic language and vice versa. In the table below we gather some basic correspondences between topological and algebraic notions:
\begin{center}
\begin{tabular}{|r|l|}
  \hline
	\textbf{Topology} & \textbf{Algebra} \\
	\hline
  Point & Character\\
  \hline
  Closed set & Ideal   \\
  \hline
	Embedding & Epimorphism \\
	\hline
	(Continuous) Surjection & Monomorphism \\
	\hline
	Homeomorphism & Automorphism \\
	\hline
	Disjoint sum & Direct sum \\
	\hline
	Cartesian product & Tensor product \\
	\hline
	Connectedness & Lack of nontrivial projections \\
	\hline
	Probabilistic measure & State \\
	\hline
\end{tabular}
\end{center}
\par

This is a reason to think about general (not necessarily commutative) $C^*$-algebras as \textit{noncommutative topological spaces}. This philosophy is a part of the much more general program called \textit{Noncommutative Geometry} (see e.g. \cite{Con, Var, Kha}).  
For our purposes we recall one more correspondence in the above spirit:
\begin{thm} Let $X$ be a compact Hausdorff space. Then the following conditions are equivalent:
\begin{enumerate}
\item the algebra $C(X)$ is separable;
\item $X$ is metrisable.
\end{enumerate}
\end{thm}

The classical spectrum $\hat{\Aa}$ (the set of characters) may be empty for a general $C^*$-algebra. There are several natural candidates for a generalisation of $\widehat{\Aa}$ for arbitrary $\Aa$: one candidate is the space of pure states on $\Aa$ (such states exist in abundance). Another is the space of all \textit{primitive} ideals of $\Aa$ (i.e. kernels of irreducible representations) or finally the space of (classes of unitary equivalence of) irreducible representations of algebra an $\Aa$. However for a generic $C^*$-algebra, the space of classes of irreducible representations has poor topology (not even $T_0$). From this reason we choose to work with genuine representations (instead of classes of unitary equivalence).

\begin{df} Let $n$ be a fixed positive integer. A $C^*$-algebra $\Aa$ is called \textit{n-homogenous} if for every irreducible representation $\pi:\Aa \to \Bb(\HHh_{\pi})$ we have $\dim \HHh_{\pi}=n$.
If instead of equality, we have inequality, i.e. $\dim \HHh_{\pi} \leq n$ then $\Aa$ is called \textit{n-subhomogeneous} or $n$-SH in short. If there is some $n \in \NNN$ such that $\Aa$ is $n$-homogenous (resp. $n$-subhomogenous) then $\Aa$ is called \textit{homogenous} (resp. \textit{subhomogenous}, or SH in short).
\end{df}

Homogenous $C^*$-algebras where characterised by Fell in 1961 in \cite{Fel} and also by Tomiyama and Takesaki. The description obtained by them uses \textit{fibre bundles}---an alternative approach can be found in \cite{Nie2}. On the other side, subhomogenous $C^*$-algebras were characterised in 1966 in \cite{Vas}. Alternative approach in terms of special category of the so called \textit{proper towers} was obtained in \cite{Nie4}. From this paper comes also the following definition:

\begin{df} A $C^*$-algebra $\Aa$ is called \textit{shrinking} if it is residually finite dimensional and satisfies the following condition: if $\{\pi_n\}_{n \in \NNN}$ is a sequence of irreducible representations of $\Aa$ with $\dim \HHh_{\pi_n} \to \oo$ then $\lim_{n \to \oo} \| \pi_n(a)\| \to 0$ for any $a \in \Aa$.
\end{df}

Recall that a $C^*$-algebra $\Aa$ is called \textit{residually finite dimensional (RFD in short}) if the set of its finite dimensional representations separates the points of $\Aa$. \par
Directly from the definition, each subhomogenous $C^*$-algebra is shrinking: the converse implication is not true in general. However, it is true for unital $C^*$-algebras, since 
$\|\pi_n(1)\|=\|I_{\HHh_{\pi_n}}\|=1$ for any irreducible representation $\pi$. 
Note also that any irreducible representation of a shrinking $C^*$-algebra is necessarily finite dimensional.

\subsection{Concave moduli of continuity}
In our work we propose the noncommutative analogue of modulus of continuity, for an element in general $C^*$-algebra. For our purposes we state the following definition:
\begin{df} Let $(X,d_x)$ and $(Y,d_Y)$ be two metric spaces. A function $\w:[0,\oo) \to [0,\oo)$ is called \textit{a modulus of (uniform) continuity} for $f:X \to Y$ if $\w(0)=0$, $\w$ is continuous, nondecreasing, concave and satisfies:
$$d_Y\big(f(x_1),f(x_2)\big) \leq w\big(d_X(x_1,x_2)\big).$$
We will use the following notation:
$$\Omega=\{w:[0,\oo) \to [0,\oo): \ \w \ \textup{is continuous, concave, nondecreasing and} \ \w(0)=0\}.$$
\end{df}

A function $f:X \to \RRR$ admits concave modulus of continuity iff $f$ is uniform limit of Lipschitz functions. This holds in particular if $(X,d)$ is compact.

\begin{rem} Concavity of $\w$ is understood as satisfying \textit{weak} inequality: $$t\w(x)+(1-t)\w(y) \leq \w(tx+(1-t)y);$$ in particular we allow constant functions to be concave thus  the zero function $0$ could serve as a modulus of continuity (for a constant function). 
\end{rem}
If $f \in C(X)$ is uniformly continuous then we can consider its minimal modulus of continuity $\w_f \in \Omega$. Then the mapping
$f \mapsto \w_f$ has the following properties:
\begin{itemize}
\item $\w_{f+g} \leq \w_f+\w_g$;
\item $\w_{cf}=|c|\w_f$;
\item $\w_{\overline{f}}=\w_f$;
\item $\w_{fg} \leq \|f\| \w_g+\|g\| \w_f$.
\end{itemize}

For further considerations in which we will define the analogue of modulus of continuity for an element in a (unital, separable) not necessarily commutative $C^*$-algebra we recall the following result (see \cite{Aro}):
\begin{thm}[Aronszajn-Panitchpakdi]\label{AP} Let $f:[0,\oo) \to [0,\oo]$ be a nondecreasing function with $f(0)=0$. Then the following conditions are equivalent:
\begin{enumerate}
\item there exists a concave, nondecreasing, continuous function $\omega:[0,\oo) \to [0,\oo)$ with $\omega(0)=0$ satisfying $f \leq \omega$;
\item there exists subadditive, nondecreasing, continuous function $\omega:[0,\oo) \to [0;\oo)$ with $w(0)=0$ satisfying $f \leq \omega$;
\item $\lim_{t \to 0^+}f(t)=0$ and $\lim \sup_{t \to \oo}\frac{f(t)}{t} < \oo$.
\end{enumerate}
\end{thm}
\begin{cor}\label{AP'} A function $\omega$ as above exists if $f$ is nondecreasing, \textit{bounded} and $f(0)=\lim_{t \to 0^+}f(t)$.
\end{cor}

\subsection{Ascoli Theorem and some remarks about convergence}
Further we will investigate the notion of compactness in the context of $C^*$-algebras therefore we recall the classical Ascoli-Arzela theorem:
\begin{thm} Let $X,Y$ be two metric spaces and assume that $X$ is compact.Then the set $K \subset C(X,Y)$ 
is relatively compact (in the uniform topology) if and only if $K$ is equicontinuous and  pointwise relatively compact.
\end{thm}

Suppose that $(X,d_X),(Y,d_Y)$ are two compact metric spaces. On the set $C(X,Y)$ of all continuous mappings $X \to Y$ we can consider the topology of pointwise convergence. Then it turns out that for a net $(u_s)_s$ in $C(X,Y)$ we have $u_s \to u$ pointwise if and only if for every function $f \in C(Y)$ we have $f \circ u_s \to f \circ u$ pointwise. Indeed: if $u_s(x) \to u(x)$ for every $x \in X$ then $f(u_s(x)) \to f(u(x))$ for every $f \in C(Y)$. Conversely, assume that for every $f \in C(Y)$ and every $x \in X$ we have $f(u_s(x)) \to f(u(x))$ and suppose that for some $x_0 \in X$ we have $u_s(x_0) \nrightarrow u(x_0)$. Passing to a subsequence (as we can, since $Y$ is compact) we can assume that there is some $y \in Y$, $y \neq u(x_0)$ such that $u_s(x_0) \to y$. But then for every $f \in C(Y)$ we conclude $f(u_s(x_0)) \to f(y)$ and $f(u_s(x_0)) \to f(u(x_0))$ hence $f(u(x_0))=f(y)$. Since continuous functions separate points we conclude $u(x_0)=y$, which yields a contradiction. \\
On $C(X,Y)$ we can also consider the topology of uniform convergence: then it turns out that $u_n \rightrightarrows u$ (where the symbol $\rightrightarrows$ means uniform convergence) if and only if for every continuous function $f \in C(Y)$ we have $f \circ u_n \rightrightarrows f \circ u$. Indeed, let $u_n \rightrightarrows u$ and $f \in C(Y)$. Then $f$ is uniformly continuous: fix $\eps >0$ and find $\delta>0$ such that for $y_1,y_2 \in Y$ satisfying $d_Y(y_1,y_2)<\delta$ the inequality $|f(y_1)-f(y_2)|<\eps$ holds. Let $s_0$ be such that for $s>s_0$ the inequality $d_Y(u_s(x),u(x))<\delta$ is satisfied, uniformly with respect to $x$: then for $s>s_0$ we conclude
$$|f(u_s(x))-f(u(x))|<\eps$$
uniformly with respect to $x$. Thus $f \circ u_n \rightrightarrows f \circ u$.
Conversely, let us assume that $f \circ u_n \rightrightarrows f \circ u$ for any $f \in C(Y)$ but it does not hold that $u_n \rightrightarrows u$.
Uniform convergence $u_n \rightrightarrows u$ is equivalent to the following condition: for every sequence $x_n \to x$, $u_n(x_n) \to u(x)$ (see for example \cite{Kur}), therefore our assumption tells us that $u_n(x_n) \nrightarrow u(x)$ for some $x \in X$ and some sequence $x_n \to x$. Since $Y$ is compact we can assume that $u_n(x_n) \to y$ for some $y$ and thus $f(u_n(x_n)) \to f(y)$ but from uniform convergence $f \circ u_n \rightrightarrows f \circ u$, we also have that $f(u_n(x_n)) \to f(u(x))$. As $f$ was arbitrary we must have $y=u(x)$. \par
In the classical context the uniform convergence $u_n \rightrightarrows u$ is more natural than the condition that $f \circ u_n \rightrightarrows f \circ u$ for every $f \in C(Y)$. 
However this second condition is better suited to noncommutative generalisations: 
if $u \in C(X,Y)$ is continuous then $u$ determine *-homomorphism $u^*:C(Y) \to C(X), \ u^*(f)=f \circ u$. Then the uniform convergence $f \circ u_n \rightrightarrows f \circ u$ is equivalent to the point-norm convergence of the sequence of *-homomorphisms $u_n^* \to u^*$.
\begin{rem} If we would define $u_n \to u$ using the condition of convergence in the operator norm $\|u_n^*-u^*\| \to 0$ we would get convergence in the discrete topology i.e. any convergent sequence $(u_n)_n$ would be eventually constant. Therefore we will not consider the convergence in the operator norm on $\Hom(\Aa,\Bb)$.
\end{rem}

\section{Motivation}
Let $X$ be a compact, metrisable space and let $d$ be a metric on $X$  inducing the original topology. Let $R:=\diam X>0$ be a (finite) diameter of the space $X$ (we assume that $X$ contains at least two distinct points). Let us define:
$$E:=\{f:X \to [0,R]: f \ \textup{is a contraction with respect to} \ d\}.$$
Then $E$ has the following properties:
\begin{enumerate}
\item $\cl{E}$ is compact: indeed, $E$ is an equicontinuous family of functions and directly from the definition is pointwise bounded. Hence applying Ascoli-Arzela theorem we infer that $E$ is relatively compact.
\item $E$ separates the points of $X$: indeed, it suffices to consider the functions $\{f_x\}_{x \in X} \subset E$ defined by $f_x(y):=d(x,y)$.
\item For each $x,y \in X$ we have a formula $$d(x,y)=\sup_{f \in E}|f(x)-f(y)|:$$
indeed, for $f \in E$ we have $|f(x)-f(y)| \leq d(x,y)$ hence $\sup_{f \in E}|f(x)-f(y)| \leq d(x,y)$. On the other hand for $f_x=d(x,\cdot)$ we obtain the equality.
\end{enumerate}
Conversely, assume that $E$ satisfies conditions (1) i (2) and define $d(x,y)$ by the formula (3). This formula indeed defines a metric:
\begin{itemize}
\item $d(x,y)$ is finite since for fixed $x,y$ the mapping
$$C(X) \ni f \mapsto |f(x)-f(y)| \in \RRR$$ is continuous, hence bounded on the compact set $\cl{E}$.
\item The condition $d(x,y)=0$ is equivalent to $f(x)=f(y)$ for $f \in E$ so from (2) we have $x=y$.
\item Symmetry is obvious and the triangle inequality follows from subadditivity of suprema.
\end{itemize}
What is more, the topology of $d$ coincides with the original topology on $X$: \\
if $d(x_n,x) \to 0$ i.e. $\sup_{f \in E}|f(x_n)-f(x)| \to 0$ then for every $f \in E$ we have $f(x_n) \to f(x)$. Assume that, on the contrary, $x_n \not\to x$---then from (sequential) compactness of $X$ there is $y \in X, y \neq x$ such that $x_{n_k} \to y$ for some subsequence $(x_{n_k})_k$. Therefore we must have $f(x_{n_k}) \to f(x)$ and simultaneously $f(x_{n_k}) \to f(y)$. Since $E$ separates the points of $X$, we must have $x=y$ yielding a contradiction. \par
Conversely, assume $x_n \xrightarrow{\textup{top}X} x$. Then for $f \in E$ we have $f(x_n) \to f(x)$. Compactness of $E$ guarantees the condition of (uniform) equicontinuity, i.e.
$$\forall_{\eps >0} \exists_{\delta >0} \forall_{x,y \in X, f \in E}\big(\rho(x,y)<\delta \Rightarrow |f(x)-f(y)|<\eps\big)$$
where $\rho$ is any metric giving the topology of $X$. Thus let us fix $\eps>0$, choose $\delta>0$ as above and let $n_0 \in \NNN$ be large enough so that $\rho(x_n,x)<\delta$. Then using equicontinuity, for any $f \in E$ we would have $|f(x_n)-f(x)|<\eps$ and hence $d(x_n,x)<\eps$. \par
Now, let $\Aa$ be an arbitrary unital $C^*$-algebra. Then the following holds:
\begin{fact} The following conditions are equivalent:
\begin{itemize}
\item $\Aa$ is separable;
\item There exists a compact set $K$ generating $\Aa$ i.e. satisfying $C^*(K \cup \{1\})=\Aa$
(where $C^*(L)$ is by definition a smallest $C^*$-algebra containing $L$).
\end{itemize}
\end{fact}
\begin{proof} Suppose that $K$ is as above. Then $K$ is separable (as a topological space) hence there exists a countable dense set $\{y_n\}_{n \in \NNN} \subset K$. The smallest unital $C^*$-algebra containing $K$ is the closure (in the norm topology) of the set of all elements of  the form $p(x_1,...,x_k,x_1^*,...,x_k^*)$ where $x_1,...,x_k \in K$, $k \in \NNN$ and $p$ is a polynomial in $2k$ free (noncommuting) variables. Then the set of all elements of the form $q(y_1,...,y_k,y_1^*,...,y_k^*)$ where $k \in \NNN$ and $q$ is a polynomial of $2k$ free variables with (complex) \textit{rational} coefficients is countable and dense in $\Aa$. \par
Conversely, assume that $\Aa$ is separable. Choose a countable dense set $\{a_n\}_{n \in \NNN}$ in the (closed) unit ball in $\Aa$ and put
$$K:=\{0\} \cup \{\frac{a_n}{n}: n \in \NNN \}.$$
Then, since $\{a_n\}_{n \in \NNN}$ is bounded, hence $K$ consists from (one) convergent sequence together with its limit, therefore is compact. Moreover we have $\lin K=\lin \{a_n: n \in \NNN \}$ and this set is dense in the unit ball of $\Aa$ so being a vector space, is dense in whole $\Aa$. Therefore it generates $\Aa$.
\end{proof}

\section{Uniform structure on the set of representations}

\subsection{Metric structure}

Let $\Aa$ be a separable, unital $C^*$-algebra and $K$ be a compact set which generates $\Aa$ (we already know that such $K$ exists). Fix two (unital) *-representations of $\Aa$, $\pi_1,\pi_2:\Aa \to \Bb(\ell^2)$ and put
\begin{equation}\label{metric}
d_K(\pi_1,\pi_2):=\sup_{a \in K} \|\pi_1(a)-\pi_2(a)\|.
\end{equation}
Denote $\Rep(\Aa):=\{\pi:\Aa \to \Bb(\ell^2): \pi \ \textup{is a *-representation}, \ \pi(1)=I\}$. The choice of $\ell^2$ as a representation space is due to the following reasons:
\begin{itemize}
\item since $\Aa$ is separable, then $\Aa$ can be \textit{faithfully} represented on a separable Hilbert space
\item obviously every two infinite dimensional separable Hilbert spaces are unitarly equivalent but $\ell^2$ has the property that for any $n \in \NNN$ we have the natural embedding $\CCC^n \hookrightarrow \ell^2$. In other words, for a finite dimensional representation $\pi$ on $\CCC^n$ we have that $\aleph_0 \odot \pi=\bigoplus_{n \in \NNN}\pi$ is a representation on $\ell^2$.
\end{itemize}
\begin{fact} \label{metryka1}
The formula \eqref{metric} defines a metric on $\Rep(\Aa)$.
\end{fact}
\begin{proof}
$d_k$ has finite values: indeed, we have
$$\sup_{a \in K}\|\pi_1(a)-\pi_2(a)\| \leq \sup_{a \in K}\Big(\|\pi_1(a)\|+\|\pi_2(a)\|\Big) \leq 2 \sup_{a \in K}\|a\|<\oo$$
since $K$ is bounded. \\
If $d_K(\pi_1,\pi_2)=0$ then for every $a \in K$ we have $\pi_1(a)=\pi_2(a)$. But $K$ generates $\Aa$ and $\pi_j, \ j=1,2$ preserve all algebraic operations; therefore $\pi_1(x)=\pi_2(x)$ for each $x \in \Aa$ and $\pi_1=\pi_2$. \\
Other conditions from the definition of metric are obvious (for triangle inequality use subadditivity of suprema).
\end{proof}

\begin{rms}
The argument given above shows not only that $d_K$ has finite values but shows also that $(\Rep(\Aa),d_K)$ is bounded and $\diam \big(\Rep(\Aa)\big) \leq 2 \diam(K)$.
Moreover the same argument applies also when $K$ is only bounded, but if we allow all possible bounded sets $K$ then the topologies of metrics $d_K$ may differ. Soon we will see that it is \textit{not} the case for a compact set $K$. \par
On the other hand, if $K$ is bounded but does not generate whole $\Aa$ then $d_K$ will only be a pseudometric.  \\
Directly from the definition we obtain the following properties:
\begin{itemize}
\item $K_1 \subset K_2 \Rightarrow d_{K_1} \leq d_{K_2}$,
\item $d_{K_1 \cup K_2}=\max(d_{K_1},d_{K_2})$,
\item $d_{K_1 \cap K_2}=\min(d_{K_1},d_{K_2})$.
\end{itemize}

\end{rms}
\begin{thm}  $(\Rep(\Aa),d_K)$ is a complete metric space.
\end{thm}
\begin{proof}
Let $(\pi_n)_{n \in \NNN}$ be a Cauchy sequence in $(\Rep(\Aa),d_K)$. Then $\sup_{a \in K}\|\pi_n(a)-\pi_m(a)\| \to 0$ hence for each $a \in K$ we have $\|\pi_n(a)-\pi_m(a)\| \to 0$. Therefore we obtain Cauchy sequences $(\pi_n(a))_{n \in \NNN}$ for each $a \in K$. Let $\Aa_0$ denotes the *-algebra generated by the set $K \cup \{1\}$. We claim that $(\pi_n(x))_{n \in \NNN}$ is a Cauchy sequence for each $x \in \Aa_0$. Each element $x \in \Aa_0$ is of the form
$$x=p(a_1,...,a_k,a_1^*,...,a_k^*)$$
for some: $k \in \NNN$, $a_1,...,a_k \in K$ and a polynomial in $2k$ free variables $p$.
Then
\begin{gather*} \\
\|\pi_n(x)-\pi_m(x)\|=\|\pi_n\big(p(a_1,...,a_k,a_1^*,...,a_k^*)\big)-\pi_m\big(p(a_1,...,a_k,a_1^*,...,a_k^*)\big)\| \\
=\|p\big(\pi_n(a_1),...,\pi_n(a_k),\pi_n(a_1)^*,...,\pi_n(a_k)^*\big)- \\
-p\big(\pi_m(a_1),...,\pi_m(a_k),\pi_m(a_1)^*,...,\pi_m(a_k)^*\big)\| \to 0
\end{gather*}
since every polynomial is uniformly continuous on compact sets. \\
Now if $x \in \Aa$ is arbitrary then $x=\lim_{k \to \oo}x_k$ for some sequence $(x_k)_k \subset \Aa_0$.
Then choosing sufficiently large $k \in \NNN$ we obtain:
$$\|\pi_n(x)-\pi_m(x)\| \leq \|\pi_n(x)-\pi_n(x_k)\|+\|\pi_n(x_k)-\pi_m(x_k)\|+\|\pi_m(x_k)-\pi_m(x)\| \leq$$
$$2\|x-x_k\|+\|\pi_n(x_k)-\pi_m(x_k)\|$$
which can be made arbitrarily small. Therefore for an \textit{arbitrary} $x \in \Aa$ the sequence $(\pi_n(x))_{n \in \NNN}$ is a Cauchy sequence---hence is convergent. We define
$$\pi(x):=\lim_{n \to \oo} \pi_n(x).$$
In this manner we obtain a *-representation of $\Aa$ (since it is the point-norm limit of *-representations). From the Cauchy condition with respect to $d_K$ we have:
\begin{equation}\label{cauchy}
\forall_{\eps >0} \exists_{N_0 \in \NNN} \forall_{n,m>N_0} \forall_{a \in K} \  \|\pi_n(a)-\pi_m(a)\| < \eps.
\end{equation}
It suffices to let $m$ go to $\oo$ in (\ref{cauchy}) to get
$$\forall_{\eps >0} \exists_{N_0 \in \NNN} \forall_{n>N_0} \ d_K(\pi_n,\pi) < \eps.$$
\end{proof}

\begin{thm} \label{punktnorm} The topology of $d_K$ coincides with the point-norm topology and the compact open topology.
\end{thm}
\begin{proof}
Suppose that $\pi_n \xrightarrow{d_K} \pi$. Then for each $a \in K$ we have $\pi_n(a) \to \pi(a)$. Analogously as above we check that for each $x \in \Aa_0$ we have $\pi_n(x) \to \pi(x)$, where $\Aa_0$ is the *-algebra generated by $K$. Then for each $x \in \Aa$ we have $x=\lim_{k \to \oo}x_k$ for some $(x_k)_k \subset \Aa_0$ hence, as before, we can estimate:
\begin{gather*} \|\pi_n(x)-\pi(x)\| \leq \|\pi_n(x)-\pi_n(x_k)\|+\|\pi_n(x_k)-\pi(x_k)\|+\|\pi(x_k)-\pi(x)\| \leq \\
\leq 2\|x-x_k\|+\|\pi_n(x_k)-\pi(x_k)\| \to 0.
\end{gather*} \par
Conversely, suppose that for every $x \in \Aa$ we have $\pi_{\sigma}(x) \to \pi(x)$. In particular this holds for each $a \in K$. Since $K$ is compact then for fixed $\eps>0$ there exists a finite $\frac{\eps}{3}$-net $\{a_1,...,a_N\} \subset K$.
In other words, for every $a \in K$ there is $i \in \{1,...,N\}$ such that $\|a-a_i\| \leq \frac{\eps}{3}$. Since the set $\{a_1,...,a_N\}$ is finite hence there exists $n_0 \in \NNN$, such that for $n>n_0$ the following is true:
$$\|\pi_n(a_i)-\pi(a_i)\| \leq \frac{\eps}{3}$$
for every $i=1,...,N$. Let us fix $a \in K$ and choose $i_0 \in \{1,...,N\}$ such that $\|a-a_{i_0}\| \leq \frac{\eps}{3}$. We obtain:
\begin{gather*}
 \|\pi_n(a)-\pi(a)\|  \leq \|\pi_n(a)-\pi_n(a_{i_0})\|+\|\pi_n(a_{i_0})-\pi(a_{i_0})\|+\|\pi(a_{i_0})-\pi(a)\|  \\
\leq 2\|a-a_{i_0}\|+\|\pi_n(a_{i_0})-\pi(a_{i_0})\| \leq \eps
\end{gather*}
uniformly with respect to $a$. Thus $d_K(\pi_n,\pi)=\sup_{a \in K}\|\pi_n(a)-\pi(a)\| \to 0$. \par
We have shown that the point-norm convergence is equivalent to convergence in $d_K$---in the proof of the fact that the point-norm convergence implies convergence in the topology of $d_K$ we did not use the fact that $K$ generates $\Aa$: in particular the same argument applies for any compact set $L \subset \Aa$---therefore the point-norm convergence implies convergence in the compact-open topology. The converse implication is always valid: hence all topologies: point-norm, compact-open and the topology of $d_K$ coincide.
\end{proof}

\begin{rem} A priori we do not know whether the point-norm topology is metrisable (in particular, whether every point has a countable neighborhood system), hence we use nets instead of ordinary sequences. The same remark applies to the compact-open topology.
\end{rem}

With every $x \in \Aa$ we can associate $\hat{x}:\Rep(\Aa) \to \Bb(\ell^2)$ defined by:
\begin{equation}
\hat{x}(\pi):=\pi(x).
\end{equation}
Then the value of $d_K(\pi_1,\pi_2)$ can be expressed as:
\begin{equation}
d_K(\pi_1,\pi_2)=\sup_{a \in K}\|\hat{a}(\pi_1)-\hat{a}(\pi_2)\|.
\end{equation}

\begin{thm}
For every $x \in \Aa$ the mapping $\hat{x}$ is uniformly continuous (with respect to $d_K$).
\end{thm}
\begin{proof}
First pick $a \in K$. Then:
$$\|\hat{a}(\pi_1)-\hat{a}(\pi_2)\|=\|\pi_1(a)-\pi_2(a)\| \leq \sup_{a \in K}\|\pi_1(a)-\pi_2(a)\|=d_K(\pi_1,\pi_2),$$
thus $\hat{a}$ is contractive (Lipschitz with Lipschitz constant $1$). We have the following facts:
\begin{itemize}
\item linear combination of Lipschitz functions is Lipschitz,
\item if $\hat{a}$ is Lipschitz then $\widehat{a^*}$ also (with the same Lipschitz constant) since:
\begin{gather*}\|\widehat{a^*}(\pi_1)-\widehat{a^*}(\pi_2)\|=\|\pi_1(a^*)-\pi_2(a^*)\|=\|\big(\pi_1(a)-\pi_2(a)\big)^*\|= \\
=\|\pi_1(a)-\pi_2(a)\|=\|\hat{a}(\pi_1)-\hat{a}(\pi_2)\|,
\end{gather*}
\item if $\hat{a},\hat{b}$ are Lipschitz then $\widehat{ab}$ also. 
\end{itemize}
It follows that for $x \in \Aa_0$, where $\Aa_0$ is the *-algebra generated by $K \cup \{1\}$, the functions $\hat{x}$ are Lipschitz, therefore uniformly continuous. Let us now take any $x \in \Aa$ and express it as $x=\lim_{k \to \oo}x_k$ where $x_k \in \Aa_0$. Fix $\eps>0$, choose $k \in \NNN$ large enough, so that $\|x-x_k\|<\frac{\eps}{3}$ and put $\delta:=\frac{\eps}{3L_k}$, where $L_k$ is Lipschitz constant for $\widehat{x_k}$. If $d_K(\pi_1,\pi_2)<\delta$, we can estimate:
\begin{gather*}
\|\hat{x}(\pi_1)-\hat{x}(\pi_2)\|=\|\pi_1(x)-\pi_2(x)\| \leq \\
\leq \|\pi_1(x)-\pi_1(x_k)\|+\|\pi_1(x_k)-\pi_2(x_k)\|+\|\pi_2(x_k)-\pi_2(x)\| \leq \\
\leq 2\|x-x_k\|+\|\pi_1(x_k)-\pi_2(x_k)\|= =2\|x-x_k\|+\|\hat{x}(\pi_1)-\hat{x}(\pi_2)\| \leq \\
\leq \frac{2\eps}{3}+L_kd_K(\pi_1,\pi_2)<\eps.
\end{gather*}
\end{proof}

The fact that $\pi$ is a *-representation (i.e. a *-homomorphism), may be expressed in terms of the mapping $x \mapsto \hat{x}$: if we define, on the set of all continuous functions defined on $\Rep(\Aa)$, the *-algebra structure in a natural manner (where all operations are defined pointwise), then the mapping $x \mapsto \hat{x}$ will become a *-homomorphism. However, it is not possible to define a norm on the set of all continuous functions on $\Rep(\Aa)$ in a natural manner to obtain $C^*$-algebra structure:
\begin{exm}
The space $\Rep(\Aa)$ is almost always nonseparable: for example let us consider $\Aa:=\CCC \oplus \CCC$
and fix a projection $P \in \Bb(\ell^2)$. Define $\pi_P$ as follows:
$$\begin{cases}
\pi_P(1,0)=P \\
\pi_P(0,1)=I-P
\end{cases}$$
and extend it  linearly on all $\Aa$. As $P$ is a projection, we infer that $\pi_P$ is indeed a *-representation. Then for two distinct projections $P$ and $Q$ we obtain:
$$1 \leq \|P-Q\|=\|\pi_P(1,0)-\pi_Q(1,0)\| \leq d_K(\pi_P,\pi_Q)$$
for every compact generating set $K$ containing $(1,0)$. Since within $\Bb(\ell^2)$ one can find uncountably many pairwise distinct (orthogonal) projections we see that $\Rep(\Aa)$ cannot be separable. It also follows that $\Rep(\Aa)$ cannot be compact, since compact metric space is always separable. In particular continuous functions on $\Rep(\Aa)$ need not to be bounded and we cannot consider the supremum norm on $C(\Rep(\Aa))$.
\end{exm}
\begin{rem} One can define a weaker topology on $\Rep(\Aa)$ in the following way: we declare that a net $\{\pi_s\}_{s \in S}$ converges to $\pi$ if for any $a \in \Aa$ and every $\xi,\eta \in \ell^2$, $\langle \pi_s(a)\xi,\eta \rangle \to \langle \pi(a)\xi,\eta \rangle$.
In other words, this is the topology defined by the family of maps 
$\{p_{a,\xi,\eta}: a \in \Aa, \xi,\eta \in \ell^2\}$ where $p_{a,\xi,\eta}(\pi):=|\langle \pi(a)\xi,\eta \rangle|$. Then $\Rep(\Aa)$ with this topology becomes separable or even more---polish space, see Chapter 4 in ~\cite{Arv} (where the space $\Rep(\Aa)$ is defined in a slightly different manner).
\end{rem}

Now suppose that we have two distinct compact, generating sets $K,L \subset \Aa$. Both topologies: of the metric $d_K$ and of the metric $d_L$ coincide with the point-norm topology. It means that those two metrics $d_K, d_L$ are equivalent. But as we have already seen, the space $\Rep(\Aa)$ is rarely compact, therefore we cannot argue that these two metrics are \textit{uniformly} equivalent (i.e. the identity mapping is uniformly continuous in both directions). But it appears that such equivalence is still valid:
\begin{thm} \label{rownowazne} Let $K,L$ be as above. Then the metrics $d_K$ and $d_L$ are uniformly equivalent.
\end{thm}
\begin{proof}
It suffices to prove that for every $\eps>0$ there exists $\delta>0$ such that for any $\pi, \pi' \in \Rep(\Aa)$ the condition $d_K(\pi,\pi')<\delta$ implies that $d_L(\pi,\pi')<\eps$ (the roles of $d_K$ and $d_L$ are symmetric). \\
So let us fix $\eps>0$ and consider $\frac{\eps}{6}$-net $\{a_1',...,a_N'\}$ for $L$. In other words the set $\{a_1',...,a_n'\}$ satisfies that $\forall_{a' \in L}\exists_{i \in \{1,...,N\}} \|a'-a'_i\|<\frac{\eps}{6}$. As $K$ generates $\Aa$ thus for every $a_i, \ i=1,...,N$, we can choose $x_i$ from the *-algebra generated by $K$ such that $\|a_i-x_i\|<\frac{\eps}{6}$. Then
$$\forall_{a' \in L}\exists_{i \in \{1,...,N\}} \|a'-x_i\| \leq \|a'-a_i\|+\|a_i'-x_i\|<\frac{\eps}{3}.$$
Therefore for any $a' \in L$ we can estimate:
\begin{gather*}
\|\pi(a')-\pi'(a')\| \leq \|\pi(a')-\pi(x_i)\| +\|\pi(x_i)-\pi'(x_i)\| +\|\pi'(x_i)-\pi'(a')\| \leq \\
\leq 2\|a'-x_i\|+\|\pi(x_i)-\pi'(x_i)\|<\frac{2\eps}{3}+\|\pi(x_i)-\pi'(x_i)\|.
\end{gather*}
It now suffices to take $\delta$ small enough so that:
$$d_K(\pi,\pi')<\delta \Rightarrow \|\pi(x_i)-\pi'(x_i)\| < \frac{\eps}{3}.$$
Such choice is possible since there are only finitely many $x_i$ and every function $\widehat{x_i}$ is uniformly continuous (with respect to $d_K$).
\end{proof}

\subsection{Moduli of continuity}

\begin{df} Fix a compact set $K$ generating $\Aa$. A set $L \subset \Aa$ is called \textit{equicontinuous} if the following condition is satisfied: for every $\eps>0$ there exists $\delta>0$ such that for any $\pi,\pi' \in \Rep(\Aa)$ satisfying $d_K(\pi,\pi')<\delta$ we have $\|\pi(x)-\pi'(x)\|<\eps$ for every $x \in L$.
\end{df}

Note that if $K_1,K_2$ are two compact generating sets and $L \subset \Aa$ is equicontinuous with respect to $d_{K_1}$ then it is also equicontinuous with respect to $d_{K_2}$: indeed, fix $\eps>0$ and choose $\delta_1>0$ which is ,,good'' for equicontinuity of $L$ with respect to $d_{K_1}$. Since the metrics $d_{K_1}$ i $d_{K_2}$ are uniformly equivalent, there exists $\delta_2>0$ such that for any two representations $\pi,\pi'$ we have the following implication: $(d_{K_2}(\pi,\pi')<\delta_2 \Rightarrow d_{K_1}(\pi,\pi')<\delta_1)$. Then $\delta_2$ is ,good'' for equicontinuity of $L$ with respect to $d_{K_2}$. In other words, we have shown that the equicontinuity does not depend from the choice of the compact generating set $K$. \par
Let us fix an equicontinuous and bounded set $L \subset \Aa$ . Define the following function:

$$f^K_L(t):=\sup\{ \|\hat{a}(\pi)-\hat{a}(\pi')\|: \pi,\pi' \in \Rep(\Aa), d_K(\pi,\pi') \leq t, a \in L\},$$ $$ f^K_L:[0,\oo) \to [0,\oo].$$

As a direct consequence of the definition we have:
\begin{itemize}
\item if $L_1 \subset L_2$ then $f^K_{L_1} \leq f^K_{L_2}$,
\item if $K_1 \subset K_2$ then $f^{K_1}_L \geq f^{K_2}_L$.
\end{itemize}

We will show that the function $f^K_L$ satisfies all the conditions from Corollary ~\ref{AP'}. \\
First, for every $a \in L$ we have
$$\|\hat{a}(\pi)-\hat{a}(\pi')\| \leq f^K_L\big(d_K(\pi,\pi')\big).$$
Since the set $L$ is equicontinuous the following holds:
\begin{equation}
\forall_{\eps>0}\exists_{\delta>0}\forall_{\pi,\pi' \in \Rep(\Aa)} \forall_{a \in L} \big(d_K(\pi,\pi')<\delta \Rightarrow \|\hat{a}(\pi)-\hat{a}(\pi')\|<\eps \big),
\end{equation}
so if $d_K(\pi,\pi') \to 0$ then $\|\hat{a}(\pi)-\hat{a}(\pi')\| \to 0$ uniformly with respect to $a \in L$. It shows that $\lim_{t \to 0^+}f^K_L(t)=0$. \\
Further, since $L$ is bounded, we infer 
$$\|\hat{a}(\pi)-\hat{a}(\pi')\|=\|\pi(a)-\pi'(a)\| \leq 2\|a\| \leq 2\diam L$$
independently from the choice of $a \in L$ and of $\pi,\pi'$ hence $f^K_L$ is bounded (by $2 \diam L$). \\
Moreover $f^K_L$ is nondecreasing, since if $t_1<t_2$ then in the definition of $f^K_L(t_2)$ we just take supremum over larger set.
We showed that $f^K_L$ satisfies all conditions from Corollary ~\ref{AP'}. Therefore there exists concave, continuous, nondecreasing function $\tilde{\omega}^K_L$ satisfying $\tilde{\omega}^K_L(0)=0$ and $f^K_L \leq \tilde{\omega}^K_L$. We put
$\omega^K_L(t):=\inf\{\w(t): \w \in \Omega, \w \geq f^K_L\}$ (this definition is correct since the set over which we take infimum is nonempty, as we have shown).
In particular for a set consisting from one point $L=\{a\}$ we use the notation $f^K_a$, $\tilde{\omega}^K_a$ and $\w^K_a$.
Let $K,K'$ be two compact generating sets. We already know that $d_{K_1}$ and $d_{K_2}$ are uniformly equivalent. However the following, stronger result is valid:
\begin{thm} There exists concave, continuous, nondecreasing function $\w$ satisfying $\w(0)=0$ and
$$d_{K'} \leq \w \circ d_K.$$
\end{thm}
\begin{proof}
Define $\w=\w^K_{K'} \in \Omega$: then $f^K_{K'} \leq \w$. In particular, for every pair of representations $\pi,\pi' \in \Rep(\Aa)$ we have
$$f^K_{K'}\big(d_K(\pi,\pi')\big) \leq \w\big(d_K(\pi,\pi')\big).$$
However
\begin{gather*}
f^K_{K'}\big(d_K(\pi,\pi')\big)=\sup\{\|\underbrace{\hat{a}(\pi_1)-\hat{a}(\pi_2)}_{=\pi_1(a)-\pi_2(a)}\|: d_K(\pi_1,\pi_2) \leq d_K(\pi,\pi'), a \in K'\} \geq \\ \geq \sup\{\|\hat{a}(\pi)-\hat{a}(\pi')\|: a \in K'\}=d_{K'}(\pi,\pi'),
\end{gather*}
therefore $d_{K'} \leq w \circ d_K.$
\end{proof}

\begin{thm} \label{omegi} Using the above notation we have $\w^{K'}_a \leq \w^K_a \circ \w^{K'}_K$.
\end{thm}
\begin{proof}
It suffices to show that the function $\w^K_a \circ \w^{K'}_K$ belongs to $\Omega$ and dominates $f^{K'}_a$. The composition of continuous/concave/nondecreasing functions has the same property and $\w^K_a \circ \w^{K'}_K(0)=\w^K_a(0)=0$. So it remains to show the appropriate inequality. We claim that:
\begin{equation}\label{aux}
f^{K'}_a \leq f^K_a \circ f^{K'}_K
\end{equation}
which is enough to end the proof, since we would get $f^{K'}_a \leq \w^K_a \circ \w^{K'}_K$ and the theorem follows by taking supremum on the left hand side. \\
For the proof of (\ref{aux}) we rewrite both sides as follows:
\begin{gather*}
L(t)=\sup\{\|\hat{a}(\tau)-\hat{a}(\tau')\|: d_{K'}(\tau,\tau') \leq t\} \\
R(t)=f^K_a\bigg(\sup\{\|\hat{b}(\pi)-\hat{b}(\pi')\|: b \in K, d_{K'}(\pi,\pi') \leq t \bigg)= \\
\sup\bigg\{\|\hat{a}(\tau)-\hat{a}(\tau')\|:d_K(\tau,\tau') \leq \sup\{|\hat{b}(\pi)-\hat{b}(\pi')\|: b \in K, d_{K'}(\pi,\pi') \leq t\}\bigg\}
\end{gather*}
The set over which we take supremum on the left hand side is contained in the set over which we take supremum on the right hand side: indeed, suppose that $\tau,\tau'$ satisfy
$d_{K'}(\tau,\tau') \leq t$. Then
$$d_K(\tau,\tau')=\sup\{\|\hat{b}(\tau)-\hat{b}(\tau')\|: b \in K\} \leq$$  $$\leq \sup\{\|\hat{b}(\pi)-\hat{b}(\pi')\|: b \in K, d_{K'}(\pi,\pi') \leq t\}.$$
\end{proof}
\begin{rem}
The above argument is valid if we replace $\{a\}$ by any equicontinuous, bounded set $L$. In other words, the following holds:
\begin{equation}\label{cykl}
\w_L^{K'} \leq w_L^K \circ w_K^{K'}.
\end{equation}
\end{rem}

\begin{prp} If $K$ is a compact set generating $\Aa$ and $a,b \in \Aa, \lambda \in \CCC$ then:
\begin{enumerate}
\item $\w^K_{a+\lambda 1}=\w_a^K=\w_{a^*}^K$,
\item $\w^K_a=0 \iff a \in \CCC 1$,
\item $\w^K_{\lambda a}=|\lambda|\w^K_a$,
\item $\w^K_{a+b} \leq \w^K_a+\w^K_b$,
\item $\w^K_{ab} \leq \|a\|\w^K_b+\|b\|\w^K_a$.
\end{enumerate}
\end{prp}
\begin{proof}
Ad. 1. For any $\pi,\pi'$ we have
\begin{gather*}
\|\pi(a^*)-\pi'(a^*)\|=\|\big(\pi(a)-\pi'(a)\big)^*\|=\|\pi(a)-\pi'(a)\|, \\
\|\pi(a+\lambda 1)-\pi'(a+\lambda 1)\|=\|\pi(a)+\lambda I-\pi'(a)-\lambda I\|=\|\pi(a)-\pi'(a)\|,
\end{gather*}
which implies that $f^K_{a+\lambda 1}=f^K_{a^*}=f^K_a$ giving (1). \par
Ad. 2. If $a=\lambda 1$ then obviously $\w^K_a=0$. \\
Conversely, assume $\w^K_a=0$: then for any $\pi,\pi' \in \Rep(\Aa)$ we have
$\pi(a)=\pi'(a)$. Fix $\pi$ and put $\pi'(\cdot):=U^*\pi(\cdot)U$ where $U \in \Bb(\ell^2)$ is a unitary. In this way we obtain $\pi(a)=U^*\pi(a)U$ and $U\pi(a)=\pi(a)U$. Since every operator in $\Bb(\ell^2)$ is a linear combination of (at most) four unitaries (see. e.g. \cite{Sak}) then $\pi(a) \in \Zz\big(\Bb(\ell^2)\big)=\CCC I$. It follows that for any representation $\pi \in \Rep(\Aa)$ there exists a number $\lambda_{\pi} \in \CCC$ such that $\pi(a)=\lambda_{\pi}I$. Therefore, for a \textit{faithful} representation $\pi_0 \in \Rep(\Aa)$ we have
$$\pi_0(a)=\lambda_{\pi_0}I=\pi_0(\lambda_{\pi_0}1),$$
and thus $a=\lambda_{\pi_0}1$ (and in fact the constant $\lambda_{\pi}$ does not depend from the choice of representation). \par
Ad. 3. Since
\begin{gather*}
\sup\{\|\pi(\lambda a)-\pi'(\lambda a)\|: d_K(\pi,\pi') \leq t\}=
\sup\{|\lambda| \|\pi(a)-\pi'(a)\|: d_K(\pi,\pi') \leq t\}= \\
=|\lambda| \sup\{\|\pi(a)-\pi'(a)\|: d_K(\pi,\pi') \leq t\},
\end{gather*}
hence $f^K_{\lambda a}=|\lambda|f^K_a$, which gives (3). \par
Ad. 4. We have
\begin{gather*}
\sup\{\|\pi(a+b)-\pi'(a+b)\|: d_K(\pi,\pi') \leq t\} \leq \\
\leq  \sup\{\|\pi(a)-\pi'(a)\|+\|\pi(b)-\pi'(b)\|: d_K(\pi,\pi) \leq t\} \leq \\
\leq \sup\{\|\pi(a)-\pi'(a)\|: d_K(\pi,\pi') \leq t \}+\sup\{\|\pi(b)-\pi'(b)\|: d_K(\pi,\pi') \leq t \},
\end{gather*}
therefore $f^K_{a+b} \leq f^K_a+f^K_b \leq \w^K_a+\w^K_b$. Since the sum of two nondecreasing/continuous/concave and vanishing at $0$ functions also has this property then $\w^K_{a+b} \leq \w^K_a+\w^K_b$. \par
Ad. 5. We have
\begin{gather*}
\|\pi(ab)-\pi'(ab)\|=\|\pi(a)\pi(b)-\pi'(a)\pi'(b)\| \leq \\
\leq \|\pi(a)\pi(b)-\pi(a)\pi'(b)\|+
\|\pi(a)\pi'(b)-\pi'(a)\pi'(b)\| \leq \\
\leq \|a\| \|\pi(b)-\pi'(b)\|+\|b\|\|\pi(a)-\pi'(a)\|,
\end{gather*}
thus $f^K_{ab} \leq \|a\|f^K_b+\|b\|f^K_a \leq \|a\|\w^K_b+\|b\|\w^K_a$ and as before,
$\|a\|\w^K_b+\|b\|\w^K_a \in \Omega$ which ends the proof.
\end{proof}

The modulus of continuity $\w_a^K$ constructed above can be identified with the minimal modulus of continuity for the function $\hat{a}$:
\begin{thm}
For $a \in \Aa$, $\w_a^K=\w^K_{\hat{a}}$.
\end{thm}
\begin{proof}
For the proof of $\w_a^K \leq \w_{\hat{a}}^K$ it is enough to show that $f^K_a \leq \w_{\hat{a}}^K$. Fix $t \geq 0$ and $\pi,\pi \in \Rep(\Aa)$ such that $d_K(\pi,\pi') \leq t$. Then we have:
\begin{equation}\label{auxx}
\|\pi(a)-\pi'(a)\|=\|\hat{a}(\pi)-\hat{a}(\pi')\| \leq \w_{\hat{a}}^K\big(d_K(\pi,\pi')\big) \leq \w_{\hat{a}}^K(t).
\end{equation}
Taking supremum in \eqref{auxx} we get $f_a^K(t) \leq \w^K_{\hat{a}}(t)$. \par
On the other hand:
$$\|\pi(a)-\pi'(a)\| \leq f^K_a\big(d_K(\pi,\pi')\big) \leq w^K_a\big(d_K(\pi,\pi')\big),$$
and therefore, from minimality, $\w^K_{\hat{a}} \leq \w^K_a$.
\end{proof}

\section{Compactness and Ascoli property}
\subsection{Definitions, properties and examples}
Let $\pi$ be an \textit{irreducible} representation of a separable $C^*$-algebra $\Aa$: then either $\HHh_{\pi}$ is finite dimensional or $\HHh_{\pi}$ is separable infinite dimensional. Therefore using the suitable unitary we can assume that $\HHh_{\pi}=\ell^2$ or $\HHh_{\pi}=\CCC^n$.
In this second case, we can consider $\pi^{\oo}:=\aleph_0 \odot \pi$. By a slight abuse of terminology (in cases it will not lead to confusion) $\pi^{\oo}$ would be also called \textit{irreducible} representation, if $\pi$ was irreducible---in this manner all irreducible representations of $\Aa$ can be realised as elements of $\Rep(\Aa)$.
We put
$$\Sigma_f(\Aa):=\cl{\{\pi^{\oo}: \pi \ \textup{is irreducible, finite dimensional}\}}$$
$$\Sigma_{\oo}(\Aa):=\cl{\{\pi: \Aa \to \Bb(\ell^2),\ \pi \  \textup{is irreducible}\}}.$$
\begin{equation*} \Sigma(\Aa):=\Sigma_f(\Aa)  \cup \Sigma_{\oo}(\Aa).
\end{equation*}
where the closure is taken with respect to the topology of $d_K$ for some compact generating set $K$ (or equivalently, in the point-norm topology).
For the convenience let us also introduce the following notation:
$$\Sigma^0_f(\Aa)=\{\pi^{\oo}: \pi \ \textup{is irreducible, finite dimensional}\}$$ and $$\Sigma^0_{\oo}(\Aa)=\{\pi: \Aa \to \Bb(\ell^2), \pi \ \textup{is irreducible}\}$$ and also $\Sigma^0(\Aa):=\Sigma^0_f(\Aa) \cup \Sigma^0_{\oo}(\Aa).$
Then $\Sigma(\Aa)$ becomes complete subspace of $\Rep(\Aa)$ and $\Sigma^0(\Aa)$ is dense subset of $\Sigma(\Aa)$.
\begin{df}
A $C^*$-algebra $\Aa$ is called \textit{compact} if $\Sigma(\Aa)$ is a compact topological space.
\end{df}

\begin{thm} If $\Aa$ is a unital, commutative, separable $C^*$-algebra then the spaces
$\Sigma(\Aa)$ and $\widehat{\Aa}$ are homeomorphic.
\end{thm}
\begin{proof} One checks that the mapping $\widehat{\Aa} \ni \omega \mapsto \aleph_0 \odot \omega \in \Sigma(\Aa)$ defines a homeomorphism.
\end{proof}
In particular, we get that (unital, separable) commutative $C^*$-algebras are compact.

\begin{exs} 1. Let $\Aa=C(X)$ be a commutative, separable $C^*$-algebra and let $d$ be a fixed metric on $X$. Every irreducible representation $\pi$ of $\Aa$ is one-dimensional thus is of the form $\pi(a)=\w(a)$ where $\w$ is a character of $\Aa$. Then $\pi(a)=\delta_x(a)$ for some $x \in X$. Identifying $\pi \simeq \pi^{\oo} \ (=:\pi_x)$ where $\pi^{\oo} \in \Sigma(\Aa)$ we have that
$$\pi_{x}(a)=\delta_x(a)I_{\ell^2}.$$
Therefore, if $\pi_x,\pi_y \in \Sigma(\Aa)$ are representations corresponding to points $x,y \in X$ then
$$\|\pi_x(f)-\pi_y(f)\|=\|f(x)I-f(y)I\|=|f(x)-f(y)|.$$
So if we take
$$K:=\{f:X \to [0,\diam X]: f \ \textup{is a contraction with respect to} \ d\},$$
we would get that $d_K(\pi_x,\pi_y)=d(x,y)$. In other words, the metric spaces $(X,d)$ and $(\Sigma(\Aa),d_K)$ are \textit{isometric}.
Thus the pair $(\Sigma(\Aa),d_K)$ where $K \subset \Aa$ is a compact generating set for a unital, separable $C^*$-algebra $\Aa$ may be though as a generalisation of the (classical) spectrum of a (unital, separable) $C^*$-algebra, treated as a metric space.
\\
Note also that if we drop the assumption of compactness of $K$ then the topology of $d_K$ may (drastically) differ from the point-norm topology. As an example of this, let us take $K$ to be the unit ball in $\Aa$. Then for $x,y \in X, x \neq y$ we can always find $f \in K$ such that $f(x)=0$ and $f(y)=1$. Thus we get $d_K(\pi(x),\pi(y)) \geq 1$ which implies that irreducible representations in  $\Sigma(\Aa)$ form a discrete set. \par
2. If $\Aa,\Bb$ are any two $C^*$-algebras then the set $\Irr(\Aa \oplus \Bb)$ of irreducible representations of $\Aa \oplus \Bb$ may be identified with $\Irr(\Aa) \sqcup \Irr(\Bb)$ in the following way: if $\pi:\Aa \oplus \Bb \to \Bb(\HHh)$ is an irreducible representation then either there exists irreducible representation $\pi_1:\Aa \to \Bb(\HHh)$ satisfying $\pi(a,b)=\pi_1(a)$ or there exists irreducible representation $\pi_2:\Bb \to \Bb(\HHh)$ such that $\pi(a,b)=\pi_2(b)$.
It follows that $\Sigma(\Aa \oplus \Bb)=\Sigma(\Aa) \sqcup \Sigma(\Bb)$ which shows that the direct sum of compact $C^*$-algebras is also compact.
\end{exs}

The above example is a special case of a more general statement:
\begin{thm} \label{epizwarta} If $\Aa$ and $\Bb$ are $C^*$-algebras, $\Aa$ is compact and $\alpha: \Aa \to \Bb$ is a *-epimorphism then $\Bb$ is also compact.
\end{thm}
\begin{proof} Define $\alpha^*: \Sigma(\Bb) \to \Sigma(\Aa)$ as $\alpha^*(\rho):=\rho \circ \alpha$. Then $\alpha^*$ indeed has values in $\Sigma(\Aa)$, since if $\rho$ is irreducible, then $(\rho \circ \alpha(\Aa))'=(\rho(\Bb))'=\CCC I$ thus $\alpha^*(\rho)$ is also irreducible; from the continuity of $\alpha^*$ (in the point-norm topology) we conclude that $\alpha^*(\rho) \in \Sigma(\Aa)$ for $\rho \in \Sigma(\Bb)$. \par
Take a compact set $K$ generating $\Aa$: then $\alpha(K)$ is also compact and generates $\Bb$. Therefore we get, for $\rho,\rho' \in \Sigma(\Bb)$:
\begin{gather*}
d_K(\alpha^*(\rho),\alpha^*(\rho'))=\sup\{\|\rho(\alpha(a))-\rho'(\alpha(a))\|: a \in K\}=\\
=\sup\{\|\rho(b)-\rho'(b)\|: b \in \alpha(K)\}=d_{\alpha(K)}(\rho,\rho').
\end{gather*}
In this way we obtain an isometric embedding $(\Sigma(\Bb),d_{\alpha(K)}) \hookrightarrow (\Sigma(\Aa),d_K)$ and it follows that $\Sigma(\Bb)$ is compact.
\end{proof}

\begin{cor} If $\Aa$ is a compact $C^*$-algebra and $\Ii \subset \Aa$ is an ideal then
$\Aa / \Ii$ is also compact.
\end{cor}
All $C^*$-algebras considered so far were unital, therefore the question whether $\Ii$ is compact is not well posed. To deal with it we introduce the following:
\begin{df} A nonunital $C^*$-algebra $\Aa$ is called \textit{compact} if $\Aa^+$ is compact.
\end{df}
\begin{rem} If $\Aa$ already has a unit we can still adjoin the unit to $\Aa$ and consider $\Aa^+$. In such case we have an isomorphism $\Aa^+ \cong \Aa \oplus \CCC$ (see e.g. \cite{Weg}) thus $\Sigma(\Aa^+)=\Sigma(\Aa) \sqcup \{\delta\}$ where $\delta(x,\lambda)=\lambda$ for $x \in \Aa, \lambda \in \CCC$. Therefore $\Aa$ is compact if and only if  $\Aa^+$ is compact.
\end{rem}

With such definition we have the following:
\begin{fact} If $\Ii \subset \Aa$ is an ideal in a unital, compact $C^*$-algebra $\Aa$ then $\Ii$ is also compact.
\end{fact}
\begin{proof} If $\pi$ is an irreducible representation for $\Ii$ then $\pi$ can be uniquely extended to a representation of $\Aa$ (which is obviously still irreducible). In other words, every irreducible representation of $\Ii$ is a restriction of irreducible representation of $\Aa$: this gives a surjection $\Irr(\Aa) \ni \pi \mapsto \pi|_{\Ii} \in \Irr(\Ii) \cup \{0\}$, and after adjoining a unit to $\Ii$---surjection $\Irr(\Aa) \ni \pi \mapsto \pi|_{\Ii^+} \in \Irr(\Ii^+)$. Thus we get the restriction mapping which can be viewed as a mapping $\Rep(\Aa) \to \Rep(\Ii^+)$ and obviously it is continuous in the point-norm topology---therefore the above remarks show that this gives a surjection $\Sigma(\Aa) \to \Sigma(\Ii^+)$. Since $\Sigma(\Aa)$ is a compact space hence $\Sigma(\Ii^+)$ is also and $\Ii$ is compact.
\end{proof}

The above proofs show that any epimorphism of $C^*$-algebras determines an embedding on the level of spaces $\Sigma(-)$; on the other hand, an inclusion of an ideal in $C^*$-algebra produces an epimorphism at the level of spaces $\Sigma(-)$.

A rich class of compact $C^*$-algebras is provided by unital, subhomogenous $C^*$-algebras:
\begin{thm} \label{subh} If $\Aa$ is $N$-subhomogenous $C^*$-algebra then $\Aa$ is compact.
\end{thm}
\begin{proof} For a natural number $n$ define
$$X_n:=\{\pi:\Aa \to M_n: \pi \ \textup{is a unital *-representation}\}.$$
We equip $X_n$ with the topology of pointwise convergence. We claim that $X_n$ is compact. To see this, consider the mapping $$\iota: X_n \ni \pi \mapsto (\pi(a))_{a \in \Aa} \in \prod_{a \in \Aa}B\big(0,\|a\|\big) \subset \prod_{a \in \Aa}M_n$$
where the codomain is equipped with the product topology. 
This mapping is obviously injective and for $\{\pi_i\}_i \subset X_n$ we have $\pi_i \to \pi$ 
if and only if $\iota(\pi_i) \to \iota(\pi)$---therefore this mapping is an embedding. 
Moreover, this embedding is closed: to see this, assume that $z_i \to z$ where $z_i \in \iota(X_n)$---thus $z_i=(\pi_i(a))_{a \in \Aa}$. Let $z$ be of the form $z=(z_a)_{a \in \Aa}$---then for every $a \in \Aa$ we have $\pi_i(a) \to z_a$. This implies that the mapping $\Aa \ni a \mapsto z_a$ becomes a unital *-representation: for example
$$z_{ab}=\lim_i\pi_i(ab)=\lim_i (\pi_i(a)\pi_i(b))=\lim_i \pi_i(a) \lim_i\pi_i(b)=z_az_b$$
(the analogous argument applies to other algebraic operations). Denote by $\pi$ the representation obtained above: then $\pi_i \to \pi$ and hence $z=\iota(\pi) \in \iota(X_n)$. \\
From the Tychonoff theorem the cartesian product $\prod_{a \in \Aa}B\big(0,\|a\|\big)$ 
is compact thus $\iota(X_n)$ is also compact (as a closed subset) and hence $X_n$ is compact as well. \\
Since $\Aa$ is $N$-subhomogenous then $\Irr(\Aa) \subset \bigsqcup_{n=1}^NX_n$ and, as the mapping $\pi \mapsto \aleph_0 \odot \pi$ is an embedding, the set $\aleph_0 \odot \bigsqcup_{n=1}^NX_n$ is a compact subset of $\Rep(\Aa)$ (containing $\Sigma^0(\Aa)$). Taking the closure we get $\Sigma(\Aa) \subset \aleph_0 \odot \bigsqcup_{n=1}^N X_n$ hence $\Sigma(\Aa)$ is a compact space and $\Aa$ is a compact $C^*$-algebra.
\end{proof}

The above result still holds for all shrinking $C^*$-algebras:
\begin{thm} If $\Aa$ is a shrinking $C^*$-algebra then $\Aa$ is compact (usually as a nonunital $C^*$-algebra).
\end{thm}
\begin{proof} Let $\Bb=\Aa^+$. Since $\Rep(\Bb)$ and $\Sigma(\Bb)$ are metrisable spaces hence compactness is equivalent to sequential compactness. Moreover, as $\Bb$ is shrinking, $\Sigma^0_{\oo}(\Bb)=\emptyset$ thus $\Sigma^0(\Bb)=\Sigma^0_f(\Bb)$ and $\Sigma(\Bb)=\cl{\Sigma^0_f(\Aa)}$. Therefore it suffices to check sequential compactness for $\Sigma^0_f(\Aa)$. So take a sequence $(\pi_n)_n \subset \Irr(\Aa)$---if (numerical) sequence $(\dim \HHh_{\pi_n})_n$ has a bounded subsequence then this subsequence is contained in $\bigsqcup_{n=1}^NX_n$ for some $N \in \NNN$, where $X_n$ are such as in the proof of Theorem \ref{subh}. Thus arguing as in the proof of Theorem \ref{subh}, we conclude that this subsequence has a further subsequence which is convergent: call it $(\pi_{n_{k_m}})_m$. Then the sequence $(\pi_{n_{k_m}}^{\oo})_m $ is convergent in $\Sigma(\Bb)$. On the other hand, if $\dim \HHh_{\pi_n} \to \oo$ then $\pi_n^{\oo} \to \delta^{\oo}$ where $\delta(a+\lambda 1)=\lambda$ is the one dimensional (irreducible) representation: indeed, for $a \in \Aa$ we have
$$\|\pi^{\oo}_n(a+\lambda 1)-\delta^{\oo}(a+\lambda 1)\|=\|\pi_n^{\oo}(a)\| \to 0.$$
\end{proof}
In the previous section we considered the notion of equicontinuity of $L \subset \Aa$, functions $f_L^K$ defined as $f^K_L(t)=\sup \{|\pi(a)-\pi'(a)\|: a \in L, d_K(\pi,\pi') \leq t\}$, the appropriate moduli of continuity $\w^K_a$ etc. All these notions can be defined once again using not all representations but only those lying in $\Sigma(\Aa)$. Obviously, if a set $L \subset \Aa$ is equicontinuous in the sense of the definition which uses $\Rep(\Aa)$, then it is also equicontinuous in the sense of the definition using $\Sigma(\Aa)$; similarly, all moduli of continuity defined using $\Sigma(\Aa)$ will be estimated by those defined using all elements in $\Rep(\Aa)$. All the proofs stay the same when we replace $\Rep(\Aa)$ by $\Sigma(\Aa)$---some difficulties arise in the context of morphisms between $C^*$-algebras. The reason is that irreducibility do not behave \textit{functorially}: to be more precise, if $\pi$ is an irreducible representation and $\alpha$ is a *-homomorphism between $C^*$-algebras then the composition $\pi \circ \alpha$ need not to be irreducible. Note that if we have defined the notion of equicontinuity and the notion of modulus of continuity  using only irreducible representations (without taking the closure) then we would obtain equivalent definitions as if we have used $\Sigma(\Aa)$. Thus the distinction whether we use a uniform structure defined with the help of $\Irr(\Aa)$ or $\Sigma(\Aa)$ is irrelevant in contrast to what happens for $\Sigma(\Aa)$ and $\Rep(\Aa)$.

We will describe how these two are related in the commutative case. We will also give a interpretation of $\Rep(\Aa)$ in the end of the paper.  

\begin{df}
We say that a $C^*$-algebra $\Aa$ has \textit{Ascoli property}, if the following condition holds: any subset $L \subset \Aa$ is relatively compact if and only if it satisfies the conditions:
\begin{enumerate}[({A}sc1)]
\item $L$ is bounded,
\item $L$ is pointwise relatively compact, i.e. for any irreducible representation $\pi$ the set $\pi(L)=\{\pi(a): a \in L\}$ is relatively compact in $\| \cdot \|$,
\item $L$ is equicontinuous.
\end{enumerate}
We say that $\Aa$ has \textit{strong Ascoli property}, if the following condition holds: any subset $L \subset \Aa$ is relatively compact if and only if it satisfies:
\begin{enumerate}[(s{A}sc1)]
\item $L$ is bounded,
\item $L$ is equicontinuous.
\end{enumerate}
\end{df}
\begin{rms} Directly from the definition strong Ascoli property implies Ascoli property.  \\
Obviously the essence of the above definition is that those conditions are \textit{sufficient} for compactness. This is because every compact set has all these properties---boundedness and pointwise relative compactness are obvious. For the equicontinuity note that if $L$ is compact then we can consider compact generating set $K$ such that $L \subset K$. Then for any $x \in L$ and $\pi,\pi' \in \Rep(\Aa)$ we have
$$\|\pi(x)-\pi'(x)\| \leq \sup_{y \in L} \|\pi(y)-\pi'(y)\| \leq \sup_{y \in K} \|\pi(y)-\pi'(y)\|=d_K(\pi,\pi'),$$
which proves equicontinuity with respect to $\Rep(\Aa)$ and therefore also with respect to $\Sigma(\Aa)$.
\end{rms}

\begin{exs} 1. Suppose that $\Aa,\Bb$ have Ascoli property. We will show that $\Aa \oplus \Bb$ also has Ascoli property. If $K_1,K_2$ are compact generating sets for $\Aa,\Bb$ (respectively) then the set $K:=K_1 \times \{0\} \cup \{0\} \times K_2$ is compact and generates $\Aa \oplus \Bb$. Suppose that $L \subset \Aa \oplus \Bb$ is bounded, pointwise relatively compact and equicontinuous (with respect to $d_{K}$) and denote by $p_1:\Aa \oplus \Bb \to \Aa$ and $p_2:\Aa \oplus \Bb \to \Bb$ the natural projections. Then obviously $p_i(L), i=1,2$ are bounded; for every irreducible representation $\pi$ of $\Aa$ we can consider $\tilde{\pi}(a,b):=\pi(a)$---then $\tilde{\pi}$ is irreducible and the set $\tilde{\pi}(L)=\pi\big(p_1(L)\big)$ is relatively compact. The same argument applies to $p_2$. Fix $\eps>0$ and choose $\delta>0$ from the definition of equicontinuity of $L$ such that for any $\pi,\pi' \in \Irr(\Aa \oplus \Bb)$ such that $d_{K}(\pi,\pi') < \delta$ we have $\|\pi(a,b)-\pi'(a,b)\|<\eps$. Note that if $\pi_1,\pi_2 \in \Irr(\Aa)$ and $\tilde{\pi_1},\tilde{\pi_2}$ denote natural extensions of $\pi_1,\pi_2$ to $\Aa \oplus \Bb$ then the following holds:
\begin{gather*}
d_{K_1}(\pi_1,\pi_2)=\sup\{\|\pi_1(a)-\pi_2(a)\|: a \in K_1\}= \\
=\sup\{\|\tilde{\pi_1}(a,b)-\tilde{\pi_2}(a,b) \|: (a,b) \in K \}=d_{K}(\tilde{\pi_1},\tilde{\pi_2}),
\end{gather*}
therefore the same $\delta$ is good for the equicontinuity of $p_1(\Aa \oplus \Bb)$ (for $p_2$ we argue similarly). From the fact that $\Aa$ and $\Bb$ have Ascoli property we conclude that $p_1(L), p_2(L)$ are relatively compact, thus $p_1(L) \oplus p_2(L)$ also and hence $L \subset p_1(L) \oplus p_2(L)$ also. \par

2. Consider $\Aa=\Kk(\ell^2)^+$, the unital $C^*$-algebra generated by compact operators.  Then $\Aa$ is not compact: let $K$ be a compact generating set for $\Aa$, containing $\{\frac{1}{n}P_n: n \in \NNN\}$ where $P_n$ is the projection onto $\lin\{e_n\}$ where $e_n$ is the $n$-th vector from the canonical basis of $\ell^2$. Consider unitaries $U_{n,m}$ defined on the orthonormal basis by the formula:
$$U_{n,m}(e_k)=
\begin{cases}
e_k, \ k \neq n \ i \ k \neq m, \\
e_n, \ k=m, \\
e_m, \ k=n
\end{cases}$$
They satisfy $U_{n,m}=U_{n,m}^*=U_{n,m}^{-1}$. Consider irreducible representations:
$$\pi_{n,m}(A):=U_{n,m}AU_{n,m} \ (=U_{n,m}AU_{n,m}^{-1}).$$
Then for $m \neq m', m,m'>1$ we infer
\begin{gather*}
\pi_{1,m}(P_1)e_m=U_{1,m}P_1U_{1,m}e_m=U_{1,m}P_1e_1=U_{1,m}e_1=e_m, \\
\pi_{1,m'}(P_1)e_m=U_{1,m'}P_1U_{1,m'}e_m=U_{1,m'}P_1e_m=0.
\end{gather*}
Thus
$$d_K(\pi_{1,m},\pi_{1,m'}) \geq \|\pi_{1,m}(P_1)-\pi_{1,m'}(P_1)\| \geq 1,$$
hence the sequence $\{\pi_{1,n}\}_{n \geq 2}$ does not have a convergent subsequence. \\
The key in the above argument is the fact that we can consider \textit{equivalent}, but still \textit{distinct} representations: the algebra $\Kk(\ell^2)$ possess up to unitary equivalence \textit{exactly one} irreducible representation (see e.g. Corollaries from Theorem 1.4.4. in \cite{Arv}.) \par
It can be shown that $\Sigma(\Aa)$ is not only noncompact, but is not even \textit{locally compact}. Indeed, every irreducible representation $\pi \in \Rep(\Aa)$ is of the form $\pi_U(\cdot)=U(\cdot)U^{-1}$ for some $U \in \Uu(\ell^2)$. Thus $\Sigma_0(\Aa)$ is a continuous image of $\Uu(\ell^2)$ through the mapping $U \mapsto \pi_U$---since $\Uu(\ell^2)$ is connected, then $\Sigma_0(\Aa)$ also and therefore so is $\Sigma(\Aa)$. Suppose that $\Sigma(\Aa)$ is locally compact: then from the Aleksandrov's Theorem (being locally compact, connected metric space) it is necessarily separable and therefore $\sigma$-compact. Therefore the space $X$ defined as the (continuous) image of $\Sigma(\Aa)$ through the map $\pi \mapsto \pi(P_1)\xi$ (where $\xi$ is some fixed unit vector) is also $\sigma$-compact: but one can show that $X$ is a sphere in $\ell^2$. It remains to note that the sphere $S$ in $\ell^2$ is not $\sigma$-compact: if it would be the case then also the \textit{ball} $B$ in $\ell^2$ would be $\sigma$-compact, as the image through the continuous mapping $S \times [0,1] \ni (\xi,t) \mapsto t\xi \in B$. But if $B=\bigcup_{n=1}^{\oo}K_n$ where $K_n$ are compact, then from the Baire's Theorem one of $K_n$'s must have nonempty interior which is a contradiction. This argument shows that $\Sigma(\Aa)$ is not locally compact. \par
But $\Aa$ still has Ascoli property: indeed, suppose that $L$ is pointwise relatively compact and consider the (irreducible) identity representation  $\iota:\Aa \hookrightarrow \Bb(\ell^2)$. Then $L=\iota(L)$ is relatively compact. \par
3. Let $S$ be the unilateral shift operator i.e. $S:\ell^2 \to \ell^2$ is given by the formula  $S(x_1,x_2,...):=(0,x_1,x_2,...)$. Then $S$ is a nonunitary isometry---in particular $S$ is not a normal operator. Thus the $C^*$-algebra generated by $S$, $\Aa:=C^*(S)$ is noncommutative. It is not so hard to show that $\Aa$ contains all compact operators: it turns out that $\Aa$ is an extension of $C(\TTT)$ by $\Kk(\ell^2)$ i.e. we have a following short exact sequence of $C^*$-algebras (see e.g. \cite{Cob}):
$$ 0 \longrightarrow \Kk(\ell^2) \longrightarrow  \Aa \longrightarrow  C(\TTT) \longrightarrow  0.$$
Therefore, as for $\Kk(\ell^2)$, the identity representation is irreducible which allows us to conclude that $\Aa$ has Ascoli property. Moreover for a compact generating set for $\Aa$ one can take $K_1=K \cup \{S\}$ where $K$ is defined as in the previous example---then similarly as before we conclude that $\Aa$ is not compact.
\end{exs}

\begin{rem} The above example can be modified in such a way that we can get the ,,scattered'' sequence not only in the norm but also in the SOT topology (which is weaker then the norm topology). It suffices to take for example:
$$U_{n}(e_k)=
\begin{cases}
e_k, \ k \neq 1 \ i \ k \neq n, \\
\frac{1}{\sqrt{2}}(e_1-e_n) , \ k=n, \\
\frac{1}{\sqrt{2}}(e_1+e_n) \ k=1
\end{cases}$$
\end{rem}

We would like to generalize the argument given above and show that if $\Aa$ is compact then every irreducible representation of $\Aa$ is finite dimensional.
For our further purposes we introduce the following:
\begin{df}
\textit For $T \in \Bb(\HHh)$ we define its \textit{orbit} as a following set:
$$\Orb(T):=\{UTU^{-1}: U \in \Uu(\HHh)\}.$$
\end{df}

\begin{thm}
Let $\HHh$ be an infinite dimensional Hilbert space and $T \in \Bb(\HHh)$. The following conditions are equivalent:
\begin{enumerate}
\item $\cl{\Orb(T)}$ is compact in $\| \cdot \|$;
\item $\cl{\Orb(T)}^{SOT}$ is SOT-compact;
\item there exists $\lambda \in \CCC$ such that $T=\lambda I$.
\end{enumerate}
\end{thm}
\begin{proof}
If $T=\lambda I$ then $\cl{\Orb(T)}$ consists from one element. \par
Thus let us assume that $T$ is not of the form $\lambda I$. Then there exists $\xi \in \HHh$, such that $T\xi$ and $\xi$ are linearly independent. We can assume without the loss of generality that $\|\xi\|=1$. Put
$$\begin{cases}
 e_1:=\xi  \\
 e_2:=\alpha \xi+\beta T\xi \\
\end{cases}$$
where $\alpha,\beta \in \CCC$ are chosen in such a way that $\|e_2\|=1$ and $e_1 \perp e_2$. In particular $\beta \neq 0$ and thus we can compute $Te_1=ae_1+be_2$ for some $a,b \in \CCC$, $b \neq 0$. We can complete $\{e_1,e_2\}$ to an orthonormal system $\{e_n\}_{n \in \NNN}$. Let $U_n$ be the unitary exchanging $e_2$ and $e_n$. Then indeed $U_n^{*}=U_n=U_n^{-1}$ and
$U_nTU_ne_1=ae_1+be_n$ whence we get:
$$\|(U_nTU_n^{-1}-U_mTU_m^{-1})e_1\|=\|be_n-be_m\|=\sqrt{2}|b|>0.$$
This contradicts the fact that $\Orb(T)$ is relatively SOT-compact. \\
The remaining part follows from the fact that SOT topology is weaker than the norm topology.
\end{proof}

\begin{rms} 1. Of course if $\dim{\HHh}< \oo$ then both topologies: the norm topology and the SOT topology coincide. Moreover, as $\Orb(T)$ is always a bounded set, thus, being finite dimensional, $\cl{\Orb(T)}$ is compact for every $T \in \Bb(\HHh)$. Note also that the mapping:
$$ \Uu(\HHh) \ni U \mapsto UTU^{-1} \in \Bb(\HHh)$$
is continuous and maps the compact topological group $\Uu(\HHh)$ onto $\Orb(T)$. In particular $\Orb(T)$ is automatically closed. \par
2. Note that the above proof shows slightly more: if $T$ is not of the form $\lambda I$ then there exists a vector $\xi \in \HHh$ such that the set $\Orb(T)\xi \subset \HHh$ is not relatively compact. But in fact the following result, which proof reminds the proof of Banach-Alaoglu theorem, is valid: (see also e.g. the proof of Theorem \ref{subh}):
\begin{thm}
A set $\SSs \subset \Bb(\HHh)$ is relatively SOT-compact if and only if for any $\xi \in \HHh$ the set $\SSs\xi$ is relatively compact in $(\HHh,\|\cdot\|)$.
\end{thm}
3. If $\dim \HHh=\oo$, $\HHh$ is separable and $T \in \Bb(\HHh)$ is normal then the following result is true (see e.g.. Theorem 1.1. in \cite{She} and references therein):
\begin{thm}
With the above assumptions, the following conditions are equivalent:
\begin{itemize}
\item $S \in \cl{\Orb(T)}$,
\item $S$ is normal, $\sigma(S)=\sigma(T)$ and for any isolated eigenvalue $\lambda \in \sigma(S)$ we have:
$$\dim \ker(T-\lambda I)=\dim \ker(S-\lambda I).$$
\end{itemize}
\end{thm}
Therefore we see that in general $\Orb(T)$ need not to be closed. \par
4. There exist several different topologies on $\Bb(\HHh)$ like for instance:  ultraweak topology (also called $\sigma$-weak), ultrastrong (also called $\sigma$-strong) and also *-ultrastrong and *-ultraweak (see e.g. \cite{Bla}, \cite{Sak},). It turns out that on the set $\Uu(\HHh)$ all these topologies coincide and are equal to the SOT (or equivalently WOT) topology (see \cite{Bla}, Prop. I.3.2.9). However the closures with respect to these topologies may differ: we have $$\cl{\Uu(\HHh)}^{\| \cdot \|}=\Uu(\HHh), $$ $$\cl{\Uu(\HHh)}^{SOT}=\{T \in \Bb(\HHh): T \ \textup{is an isometry}\},$$ $$\cl{\Uu(\HHh)}^{WOT}=\{T \in \Bb(\HHh): \|T\| \leq 1\}.$$
\end{rms}
\begin{thm}\label{zwarta} If $\Aa$ is a compact $C^*$-algebra and $\pi:\Aa \to \Bb(\HHh_{\pi})$ is an irreducible representation then $\dim \HHh_{\pi}<\oo$.
\end{thm}
\begin{proof}
For $U \in \Uu(\HHh_{\pi})$ denote $\pi_U:=U\pi U^{-1}$. This formula defines an irreducible representation. Since $\Aa$ is compact then $\Sigma(\Aa)$ is a compact space thus, in particular, the family $\{\pi_U(a): U \in \Uu(\HHh_{\pi})\}$ is relatively compact for each $a \in \Aa$. But if we put $A:=\pi(a)$ then this family is equal to $\Orb(A)$ thus it must be satisfied $\dim \HHh_{\pi}<\oo$.
\end{proof}

\begin{cor} If $\Aa$ is an infinite dimensional and simple $C^*$-algebra then $\Aa$ has Ascoli property and is not compact.
\end{cor}
\begin{proof} Since $\Aa$ is simple its every (nonzero) representation is faithful: in particular every irreducible representation is faithful---therefore $\Aa$ has Ascoli property. On the other hand, since $\Aa$ is infinite dimensional, hence does not have a finite dimensional (faithful) irreducible representation and therefore it cannot be compact.
\end{proof}

\begin{exm} We will describe a nontrivial example of a simple infinite dimensional $C^*$-algebra, defined via generators and relations. This algebra is called \textit{noncommutative torus} or \textit{(irrational) rotation algebra} and is usually denoted by $A_{\theta}$, where $\theta \in [0,1]$ is a (irrational) parameter. The algebra $A_{\theta}$ is defined as the \textit{universal} unital $C^*$-algebra generated by two unitaries $u,v$ satisfying $vu=e^{2\pi i\theta}uv$. Universality is understood as follows: whenever $\Aa$ is another unital $C^*$-algebra, containing two unitaries $U,V \in \Aa$ which satisfy $VU=e^{2\pi i \theta}UV$ then there is a morphism $\alpha:A_{\theta} \to \Aa$ such that $\alpha(u)=U, \alpha(v)=V$. If $\theta=0$ we obtain the $C^*$-algebra isomorphic to the algebra of continuous functions on two-torus $C(\TTT^2)$---this fact justifies the name ,,noncommutative torus''. For $\theta=\frac{p}{q}$ (where $p,q$ are relatively prime, $q>0$) there exists (complex) vector bundle $E$ of rank $q$ such that $A_{\theta} \cong \Gamma(\TTT^2,\textup{End}(E))$ (see e.g.\cite{Kha}, Prop. 1.1.1)---then $A_{\theta}$, although is noncommutative, is quite close to a commutative algebra (the precise notion is the \textit{Morita equivalence}---see e.g. \cite{Rae}, Chapter 3). Thus the most interesting case is when $\theta$ is irrational. Then it turns out that $A_{\theta}$ is a simple algebra (see Thm. VI.1.4 in \cite{Dav}). $A_{\theta}$ can be represented on $L^2(\SSS^1)$ as follows:
$$(Uf)(t)=e^{2\pi t}f(t), \ \  Vf(t)=f(t+\theta)$$
where the circle $\SSS^1$ is viewed as $\RRR / \ZZZ$ (to keep additive notation). In the light of the above results, $A_{\theta}$ (for irrational $\theta$) is not compact but has Ascoli property. More about $A_{\theta}$ and its various incarnations can be found in \cite{Kha}.
\end{exm}

\begin{exm} There are examples of unital $C^*$-algebras which are noncompact but such that every irreducible representation is finite dimensional. Below we give an example of such algebra. Put $\Aa:=\prod_{n=1}^{\oo}M_{2^n}$
and consider an ideal $\Ii \subset \Aa$ defined as $\Ii:=\bigoplus_{n=1}^{\oo}M_{2^n}$. Define also $\Ff=\{(A_n)_{n=1}^{\oo} \in \Aa: A_1 \ \textup{is a diagonal matrix}, \ A_{n+1}=A_n \oplus A_n, \ n \in \NNN \}$: then $\Ff$ is a commutative $C^*$-algebra, isomorphic to $\CCC \oplus \CCC$. It follows that $\Ii \cap \Ff=0$. Let $\Aa_0=\Ff \dotplus \Ii$. Since $\Ff$ is finite dimensional then $\Aa_0$ is closed in $\Aa$ hence $\Aa_0$ is also $C^*$-algebra. We will show that $\Aa_0$ is not compact but its every irreducible representation is finite dimensional. \\
We have the following short exact sequence of $C^*$-algebras:
$$0 \longrightarrow \Ii \longrightarrow \Aa_0 \longrightarrow \Ff \longrightarrow 0.$$
Since $\Ii=\bigoplus_{n=1}^{\oo}M_{2^n}$ hence $\Irr(\Ii)=\bigsqcup_n \Irr(M_{2^n})$---however $\Irr(M_k)=\{U(\cdot)U^{-1}: U \in \Uu_n\}$. Therefore every irreducible representation of $\Ii$ is finite dimensional. As $\Ff$ is of finite dimension then $\Ff$ also has this property. Since $\Aa_0$ is an extension of $\Ff$ by $\Ii$ then $\Aa_0$ also has the property of having finite dimensional irreducible representations. \par
However $\Aa_0$ is not compact: to see this let us fix $n \in \NNN$ and for an operator $U \in \Bb(\CCC^{2^n})$ consider the representation given by $\pi_{n,U}(a):=UA_nU^{-1}$ where $a=(A_n)_{n \in \NNN}, \ A_n \in M_{2^n}(\CCC)$. Viewing $\pi_{n,U}$ as a representation of $\Ii$ we conclude that it is irreducible. Therefore if we take $\pi_{n,U}$ defined by the same formula but on the whole $\Aa_0$ we also get an irreducible representation. Thus $\rho_n:=\aleph_0 \odot \pi_{n,U} \in \Sigma(\Aa_0)$.
Consider the matrix $V=
\left[\begin{array}{cccccccc}
0 & 1 \\
1 & 0  \\
\end{array}\right]$. Then $V$ is a real symmetric matrix satisfying $V^2=I$, so $V$ is a unitary and $V^{-1}=V$.
If $A=\left[\begin{array}{cccccccc}
a_{11} & a_{12} \\
a_{21} & a_{22}  \\
\end{array}\right]$,
then $VAV^{-1}=\left[\begin{array}{cccccccc}
a_{22} & a_{21} \\
a_{12} & a_{11}  \\
\end{array}\right]$.

Consider unitary operators on $C^{2^n}$ defined by the formula:
$$U_n=\underbrace{I_2 \oplus I_2 \oplus ... \oplus I_2}_{2^{n-1}-1} \oplus V$$
and take the associated irreducible representations $\pi_{n,U_n}$. As a compact set $K$ generating $\Aa_0$ take any compact set containing $a=(A,A \oplus A,A \oplus A \oplus A \oplus A,...) \in \Ff$ where for example $A=\left[\begin{array}{ccc}
1 & 0 \\
0 & -1 \\
\end{array}\right]$.
Assume for example that $n<m$. Then we have:
\begin{gather*}
d_K(\rho_n,\rho_m) \geq \|\rho_n(a)-\rho_m(a)\|=\|\aleph_0 \odot U_n(A^{\oplus 2^{n-1}})U_n-\aleph_0 \odot U_m(A^{\oplus 2^{m-1}})U_m\|= \\
=\| \aleph_0 \odot \Big((0_{2^n-2} \oplus B) \oplus (0_{2^n-2} \oplus B) \oplus ... \oplus (0_{2^n-2} \oplus B) \oplus 0_{2^n} \Big) \|= \\
=\|(0_{2^n-2} \oplus B)\|=\|B\|
\end{gather*}
where $B=VAV^{-1}-A=\left[\begin{array}{cccccccc}
-2 & 0 \\
 0 & 2 \\
\end{array}\right].$
In this way we obtain infinite discrete subset of $\Sigma(\Aa_0)$: thus $\Aa_0$ cannot be compact.
\end{exm}

\begin{rem}
The ideal $\Ii$ above is an example of shrinking (nonunital) $C^*$-algebra, therefore is compact. Obviously $\Ff$ is also compact (being commutative or being finite dimensional algebra).
Therefore the above example shows that compactness is not preserved by taking extensions: in fact it shows also that the property of being shrinking is not closed under extensions.
\end{rem}

Before we prove a theorem about the relation between compactness and strong Ascoli property, we will need the following:
\begin{lem} Suppose that $X,Y$ are metric spaces, $X$ is compact and $Y$ is complete and let $K \subset C(X,Y)$ be an equicontinuous family. Then the set $$F:=\{x \in X: K(x)-\textup{relatively compact}\}$$ is closed where $K(x):=\{f(x): f \in K\}$.
\end{lem}
\begin{proof} Since $Y$ is complete hence relative compactness is equivalent to complete boundedness so $F=\{x \in X:K(x)-\textup{complete bounded}\}$. Fix $x \in \cl{F}$ and $\eps>0$. From equicontinuity there exists $\delta>0$ such that for any function $f \in K$ and every $a,b \in X$ such that $d_X(a,b)<\delta$ we have $d_Y(f(a),f(b))<\frac{\eps}{4}$. As $x \in \cl{F}$ then there is $a \in F$ with $d_X(x,a)<\delta$. Since $a \in F$ then $K(a)$ is completely bounded thus there exist $z_1,...,z_N \in Y$ with the property:
$$K(a) \subset \bigcup_{j=1}^N B(z_j,\frac{\eps}{4}).$$
We claim that
\begin{equation}\label{kule}
K(x) \subset \bigcup_{j=1}^N B(z_j,\frac{\eps}{2}).
\end{equation}
Let $f \in K$: as $d_X(x,a)<\delta$ then $d_Y(f(x),f(a))<\frac{\eps}{4}$; moreover $f(a) \in K(a)$ so there is $j \in \{1,...,N\}$ such that $f(a) \in B(z_j,\frac{\eps}{4})$. From this we conclude:
$$d_Y(f(x),z_j) \leq d_Y(f(x),f(a))+d_Y(f(a),z_j) <\frac{\eps}{4}+\frac{\eps}{4}=\frac{\eps}{2}$$
which shows ~\eqref{kule}. In this way we obtain an $\frac{\eps}{2}$-net in $Y$ which can be modified to an $\eps$-net in $K(x)$ by defining $y_j$ as any  point from the set $K(x) \cap B(z_j,\frac{\eps}{2})$ (if $K(x) \cap B(z_j,\frac{\eps}{2})=\emptyset$ then we can skip the ball $B(z_j,\frac{\eps}{2})$ and still $K(x) \subset \bigcup_{i \neq j}B(z_i,\frac{\eps}{2})$. Thus without a loss of generality we can assume that for any $j=1,...,N$ we have $K(x) \cap B(z_j,\frac{\eps}{2}) \neq \emptyset$). Then for any element $f(x) \in K(x)$ where $f \in K$ there is $z_j \in Y$ such that $d_Y(f(x),z_j)<\frac{\eps}{2}$ and thus:
$$d_Y(f(x),y_j) \leq d_Y(f(x),z_j)+d_Y(z_j,y_j)<\frac{\eps}{2}+\frac{\eps}{2}=\eps$$
which shows that $K(x) \subset \bigcup_{j=1}^N B(y_j,\eps)$ hence $\{y_1,...,y_N\}$ is an $\eps$-net.
\end{proof}

\begin{cor} \label{wniosek1} If (using the notation and assumptions as above) the set $K(x)$ is relatively compact for every $x$ from some dense set $D \subset X$ then the same is true for every $x \in X$.
\end{cor}

\begin{thm} If $\Aa$ is a compact $C^*$-algebra then $\Aa$ has strong Ascoli property.
\end{thm}
\begin{proof}
Let $L \subset \Aa$ be bounded and equicontinuous.
For every $x \in L$ consider $\hat{x}$ as a mapping $\hat{x}: \Sigma(\Aa) \to \Bb(\ell^2)$. Then it is a (uniformly) continuous mapping, defined on a \textit{compact} metric space (and with values in a metric space)---in such circumstances we can use (classical) Ascoli theorem. We will check that the assumptions of this theorem are fulfilled:
\begin{enumerate}
\item since $L$ is equicontinuous hence the family $\{\hat{x}: x \in L\}$ is equicontinuous,
\item let $\pi \in \Sigma(\Aa)$ be of the form $\aleph_0 \odot \pi'$ where $\pi'$ is irreducible. As $\Aa$ is a compact $C^*$-algebra, $\pi'$ is finite dimensional, say $\pi': \Aa \to M_n$. Thus
$$\{\hat{x}(\pi): x \in L\}=\{\pi(x): x \in L\}=\{\aleph_0 \odot \pi'(x): x \in L\}.$$
Clearly the mapping $M_n \ni A \mapsto \aleph_0 \odot A \in \Bb(\ell^2)$ is an isometry, thus it is enough to show that $\{\pi'(x): x \in L\}$ is relatively compact---but this is a bounded subset of $\CCC^n$ thus it is relatively compact.
\end{enumerate}
Since $\Sigma^0_f(\Aa)$ is dense in $\Sigma_f(\Aa)=\Sigma(\Aa)$ hence the set $\pi(L)$ is relatively compact for any $\pi \in \Sigma(\Aa)$ (Corollary ~\ref{wniosek1}).
Therefore from Ascoli theorem the set $\{\hat{x}: x \in L\}$ is relatively compact. As the mapping $x \mapsto \hat{x}$ is an isometry (it was checked for example in the proof of Theorem \ref{uniform}), $L$ also must be relatively compact.
\end{proof}

\begin{rem} Let us note that without knowing Theorem \ref{zwarta} (and Corollary ~\ref{wniosek1}) one can show quite similarly a weaker result that every compact $C^*$-algebra has Ascoli property.
\end{rem}

\subsection{Pointwise and uniform convergence}
\begin{df} We say that a net $(a_s)_s \subset \Aa$ tends to $a \in \Aa$ \textit{pointwise}, if for any representation $\pi \in \Sigma(\Aa)$ we have $\|\pi(a_s)-\pi(a)\| \to 0$.
\end{df}

\begin{rms} It is clear that if $\Aa=C(X)$ is a commutative $C^*$-algebra then every representation belonging to $\Sigma(\Aa)$ is automatically irreducible and is of the form $f \mapsto f(x)$ for some $x \in X$. Therefore the above defined convergence is just the ordinary pointwise convergence. On the other hand, if in the above definition we require convergence for all elements in $\Rep(\Aa)$ instead of $\Sigma(\Aa)$ then we would end up with the definition of the norm convergence: this is simply due to the fact, that we could consider a faithful representation.
\end{rms}

The above defined convergence comes from a locally convex topology, determined by the family of seminorms defined by the formula:
$a \mapsto \|\pi(a)\|, \ \pi \in \Sigma(\Aa)$.

All $C^*$-algebra operations (linear operations, multiplication, involution) are continuous with respect to this topology---for example, to check the continuity of multiplication note that if $a_s \to a,\ b_s \to b$ pointwise and $\pi \in \Sigma(\Aa)$, then we have:
\begin{gather*}
\|\pi(a_sb_s)-\pi(ab)\| \leq \|\pi(a_s)\pi(b_s)-\pi(a)\pi(b_s)\|+\|\pi(a)\pi(b_s)-\pi(a)\pi(b)\| \leq \\
\|\pi(a_s)-\pi(a)\| \cdot \|\pi(b_s)\|+\|\pi(a)\|\|\pi(b_s)-\pi(b)\| \to 0
\end{gather*}
as the net $\{\|\pi(b_s)\|\}_s$ is bounded for sufficiently large $s$. \\
In the context of (not necessary commutative) compact $C^*$-algebras we have a following version of the classical Dini's Theorem:
\begin{thm} Let $\Aa$ be a compact, unital $C^*$-algebra and $(a_n)_n$ be an increasing sequence converging pointwise to $a$. Then $a_n \to a$ in the norm.
\end{thm}

\begin{proof} Denote for convenience $b_n:=a-a_n$---then $b_n \geq 0$ and $b_n \to 0$ pointwise. Fix $\eps>0$ and let $K_n=\{\pi \in \Sigma(\Aa): \|\pi(b_n)\| < \eps \}$.
As $\big(\pi(b_n)\big)_n$ is a decreasing sequence of nonnegative elements, $\big(\|\pi(b_n)\|\big)_n$ is decreasing, thus $\{K_n\}_n$ is an increasing sequence of open sets. From our assumption we have $\Sigma(\Aa)=\bigcup_{n=1}^{\oo}K_n$. Since $\Sigma(\Aa)$ is compact, there exists $N \in \NNN$ such that $K_n=\Sigma(\Aa)$ for every $n>N$. Thus if $n>N$, then for any $\pi \in \Sigma(\Aa)$ we have $\|\pi(b_n)\| <\eps$ so $\sup_{\pi \in \Sigma(\Aa)}\|\pi(b_n)\|=\|b_n\|<\eps$, which shows the convergence in the norm.
\end{proof}

Next we would like to describe the relation between the above convergence and weak convergence (in the sense of Banach space theory).
We will need the following result (see for example \cite{Phe}):
\begin{thm}[Choquet] \label{Choquet} Let $E$ be a Banach space and $K \subset E^*$ be a compact, convex and metrisable subset of the dual space. Then any $\psi \in K$ is of the form
$$\psi(x)=\int_{\textup{Ex}(K)}\varphi(x)d\mu(\varphi)$$
for some probabilistic measure $\mu$ on the set  $\textup{Ex}(K)$ of extreme points of $K$.
\end{thm}

\begin{lem} If $(x_n)_n \subset \Aa$ is a bounded sequence, then $x_n \to 0$ weakly if and only if for every pure state $\varphi$ we have $\varphi(x_n) \to 0$.
\end{lem}
\begin{proof}
Suppose that for every pure state $\varphi \in \PPp(\Aa)$, $\varphi(x_n) \to 0$.
If we define $f_n: \PPp(\Aa) \to \CCC$ by the formula $f_n(\varphi)=\varphi(x_n)$ then all $f_n$'s are *-weakly continuous and $\|f_n\|=\|x_n\|$ thus $(f_n)_n$ is uniformly bounded. Our assumption says that $f_n \to 0$ pointwise, thus from Lebesgue Theorem, for every probabilistic measure $\mu$ on $\PPp(\Aa)$ we have $\int_{\PPp(\Aa)}f_n(\varphi)d\mu(\varphi)=\int_{\PPp(\Aa)}\varphi(x_n)d\mu(\varphi) \to 0$. Since the set of all states $\SSs(\Aa)$ is *-weakly compact, convex and metrisable (if $\Aa$ is separable), then from the Choquet theorem every state $\psi \in \SSs(\Aa)$ is of the form
$$\psi(x)=\int_{\PPp(\Aa)}\varphi(x)d\mu(\varphi)$$
for some probabilistic measure $\mu$ on $\PPp(\Aa)$. Thus we get $\psi(x_n) \to 0$ for any state $\psi$. Since every linear, continuous functional on $\Aa$ is a linear combination of (at most four) states hence $x_n \to 0$ weakly.
\end{proof}
\begin{thm} Let $\Aa$ be a separable, unital $C^*$-algebra and $(x_n)_n \subset \Aa$ be a bounded sequence. The following conditions are equivalent:
\begin{enumerate}
\item $x_n \to 0$ weakly;
\item for any irreducible representation $\pi:\Aa \to \Bb(\HHh)$ we have $\pi(x_n) \to 0$ in the WOT topology;
\item for any representation $\pi:\Aa \to \Bb(\HHh)$ we have $\pi(x_n) \to 0$ in the WOT topology.
\end{enumerate}
\end{thm}
\begin{proof}
Assume that $x_n \to 0$ weakly, fix an arbitrary representation $\pi$ and arbitrary vectors $\xi,\eta \in \HHh_{\pi}$. Define $\varphi:\Aa \to \CCC$ by the formula
$$\varphi(x):=\langle \pi(x)\xi,\eta \rangle.$$
Then $\varphi$ is a linear and continuous functional thus $\varphi(x_n) \to 0$ which shows that $\pi(x_n) \to 0$ in the WOT topology. \par
Now assume that for any irreducible representation $\pi$, $\pi(x_n) \to 0$ in the WOT topology and fix $\varphi \in \PPp(\Aa)$. Then $\varphi$ gives rise to an irreducible representation $\pi_{\varphi}:\Aa \to \Bb(\HHh_{\varphi})$ such that
$$\varphi(x)=\langle \pi_{\varphi}(x)\xi_{\varphi},\xi_{\varphi} \rangle$$
where $\xi_{\varphi} \in \HHh_{\varphi}$ is a (unit) cyclic vector.
From our assumption $\langle \pi(x_n)\xi_{\varphi},\xi_{\varphi} \rangle \to 0$ so $\varphi(x_n)\to 0$ hence from the lemma above, $x_n \to 0$ weakly.
\end{proof}
Clearly if every irreducible representation of $\Aa$ is finite dimensional, then for such a representation $\pi$ the convergence $\pi(x_n) \to 0$ in the WOT topology is equivalent to the convergence in the norm and we get the following:
\begin{cor}
If $\Aa$ is a unital CCR-algebra then for a bounded sequence $(x_n)_n \subset \Aa$ the following conditions are equivalent:
\begin{itemize}
\item $x_n \to 0$ pointwise,
\item $x_n \to 0$ weakly.
\end{itemize}
In particular, the above equivalence takes place when $\Aa$ is a compact $C^*$-algebra. \qed
\end{cor}

We can also consider the following convergence:
\begin{df} We say that a net $(a_s)_s$ converges to $a$ \textit{uniformly} if $\|\pi(a_s)-\pi(a)\| \to 0$ uniformly with respect to $\pi \in \Sigma(\Aa)$.
\end{df}
Directly from the definition we see that the uniform convergence $a_s \to a$ is just the uniform convergence of $\widehat{a_s}$ to $\hat{a}$ as continuous functions $\Sigma(\Aa) \to \Bb(\ell^2)$. In particular we see that this convergence comes from some metric---thus we can restrict our attention to ordinary sequences instead of nets.
It turns out that such defined convergence coincides with the norm convergence:
\begin{thm} \label{uniform} A sequence $(a_n)_n \subset \Aa$ converges to $a \in \Aa$ uniformly iff $\|a_n-a\| \to 0$.
\end{thm}
\begin{proof}
From the formula $\|x\|=\sup\{\|\pi(x)\|: \pi \in \Rep(\Aa)\}=\sup\{\|\pi(x)\|: \pi \in \Irr(\Aa)\}$ (see 2.7.1 and 2.7.3 in \cite{Dix}) it follows
$$\|x\|=\sup\{\|\pi(x)\|: \pi \in \Sigma(\Aa)\},$$
hence the theorem follows.
\end{proof}

\section{Uniform continuity of morphisms}
\subsection{Two equivalent approaches}
Let $X,Y$ be two compact metric spaces and $\alpha:C(X) \to C(Y)$ be a *-homomorphism (preserving the units). Then $\alpha$ induces a continuous mapping $u:Y \to X$ such that $\alpha(f)=f \circ u$. Then for every $y_1,y_2 \in Y$ it follows:
\begin{gather*} |\alpha(f)(y_1)-\alpha(f)(y_2)|=|f(u(y_1))-f(u(y_2))| \leq \\
\w_f \big(d_X\big(u(y_1),u(y_2)\big)\big) \leq \w_f \big(w_u\big(d_Y(y_1,y_2)\big)\big)\end{gather*}
where $\w_f,\w_u$ are (minimal) moduli of continuity for $f$ and $u$ (resp.). Hence when $\w_{\alpha(f)}$ denotes the minimal modulus of continuity for $\alpha(f) \in C(Y)$, it follows that
$$\w_{\alpha(f)} \leq \w_f \circ \w_u.$$
For noncommutative considerations assume that $\Aa,\Bb$ are two separable, unital $C^*$-algebras and $K,L$ are compact generating sets for $\Aa$ and $\Bb$ (resp.).
The argument given above serves as a motivation for the following:
\begin{df} A *-homomorphism $\alpha: \Aa \to \Bb$ is called \textit{uniformly continuous}, if there exists $\w \in \Omega$ such that for any $a \in \Aa$ we have
$$\w^L_{\alpha(a)} \leq \w^K_a \circ \w.$$
We denote $\w_{\alpha}^{K,L}=\inf \{\w \in \Omega: \w^L_{\alpha(a)} \leq \w^K_a \circ \w, a \in \Aa\}.$
\end{df}

\begin{fact} The definition above does not depend from the choice of compact sets $K,L$ generating $\Aa$ and $\Bb$.
\end{fact}
\begin{proof} Suppose that $\alpha:(\Aa,K) \to (\Bb,L)$ is a uniformly continuous *-homo\-morphism and $K',L'$ are other compact sets generating $\Aa,\Bb$ (resp.). It follows that $\w_{\alpha(a)}^L \leq \w_a^K \circ \w_{\alpha}^{K,L}$ and we can estimate
\begin{equation} \w_{\alpha(a)}^{L'} \leq \w_{\alpha(a)}^L \circ \w_L^{L'} \leq \w_a^K \circ \w_{\alpha}^{K,L} \circ \w_L^{L'} \leq \w_a^{K'} \circ \w_{K'}^K \circ \w_{\alpha}^{K,L} \circ \w_L^{L'}
\end{equation}
which proves the uniform continuity of $\alpha$ as a mapping $(\Aa,K') \to (\Bb,L')$: in fact we get that $\w_{\alpha}^{K',L'} \leq \w^K_{K'} \circ w_{\alpha}^{K,L} \circ \w_L^{L'}$.
\end{proof}

If $\alpha: \Aa_1 \to \Aa_2,\beta: \Aa_2 \to \Aa_3$ are two *-homomorphisms of $C^*$-algebras and
$K_i,\ i=1,2,3$ are compact sets generating $\Aa_i$ then it follows:
$$\w_{\beta \circ \alpha(a)}^{K_3} \leq \w_{\alpha(a)}^{K_2} \circ w_{\beta} \leq \w_a^{K_1} \circ (\w_{\alpha} \circ \w_{\beta})$$
which shows that $\w_{\beta \circ \alpha} \leq \w_{\alpha} \circ \w_{\beta}$. In particular we see that the composition of uniformly continuous morphisms is again uniformly continuous.

Theorem \ref{omegi} shows that the identity morphism 
$$\id:(\Aa,K) \to (\Aa,K')$$ is uniformly continuous. In fact, the more general statement is valid:

\begin{thm} If $\alpha:\Aa \to \Bb$ is a *-homomorphism then $\alpha$ is uniformly continuous.
\end{thm}
\begin{proof} Fix two compact sets $K,L$ generating $\Aa$ and $\Bb$ respectively and let $a_0 \in \Aa$. We claim that for $t \geq 0$:
\begin{equation}\label{jednostajna}
f_{\alpha(a_0)}^L(t) \leq f^K_{a_0} \circ f^L_{\alpha(K)}(t).
\end{equation}
Fix $\tau_1,\tau_2 \in \Rep(\Bb)$ such that $d_L(\tau_1,\tau_2) \leq t$ and put $\pi_i:=\tau_i \circ \alpha \in \Rep(\Aa), \ i=1,2$. Then
$$\sup\{\|\tau_1(\alpha(a))-\tau_2(\alpha(a))\|: a \in K\} \leq $$ $$ \leq
\sup\{\|\tau(\alpha(a))-\tau'(\alpha(a))\|: a \in K, d_L(\tau,\tau') \leq t\}.$$
The left hand side of the above inequality is $d_K(\tau_1 \circ \alpha,\tau_2 \circ \alpha)=d_K(\pi_1,\pi_2)$ and the right hand side is $\sup\{\|\tau(b)-\tau'(b)\|: b \in \alpha(K), d_L(\tau,\tau') \leq t\}=f^L_{\alpha(K)}(t)$. Therefore we get that if $d_L(\tau_1,\tau_2) \leq t$ then:
\begin{equation} \label{zawieranie}
d_K(\pi_1,\pi_2) \leq f^L_{\alpha(K)}(t).
\end{equation}
Now the left hand side of inequality (\ref{jednostajna}) is equal to $\sup\{\|\tau(\alpha(a_0)-\tau'(\alpha(a_0))\|: d_L(\tau,\tau') \leq t\}$ while the right hand side is $\sup\{\|\pi(a_0)-\pi'(a_0)\|: d_K(\pi,\pi') \leq f^L_{\alpha(K)}(t)\}$. Thus (\ref{zawieranie}) shows that the set over which we take supremum on the left hand side of (\ref{jednostajna}) is contained in the set over which we take supremum on the right hand side of (\ref{jednostajna}) and for $\tau_i,\pi_i$ as above we have $\|\pi_1(a_0)-\pi_2(a_0)\|=\|\tau_1(\alpha(a_0))-\tau_2(\alpha(a_0))\|$ proving (\ref{jednostajna}). \par
It suffices to take suprema in (\ref{jednostajna}) to conclude $\w^L_{\alpha(a_0)} \leq \w^K_{\alpha(a_0)} \circ \w^L_{\alpha(K)}$. In particular $\w^{K,L}_{\alpha} \leq \w^L_{\alpha(K)}$.
\end{proof}
We can also consider an alternative definition:
\begin{df} A *-homomorphism $\alpha: \Aa \to \Bb$ is called \textit{uniformly continuous} if the mapping
$$\alpha^*:(\Rep(\Bb),d_L) \ni \pi \mapsto \pi \circ \alpha \in (\Rep(\Aa),d_K)$$
is uniformly continuous, as a mapping between metric spaces. \par
If $\alpha$ is an epimorphism we can also consider the mapping from $\Sigma(\Bb)$ to $\Sigma(\Aa)$ and require this mapping to be uniformly continuous.
\end{df}

Note that if $\alpha: \Aa \to \Bb$ is an epimorphism then the induced map $\alpha^*$ is a monomorphism as a mapping $\Sigma(\Bb) \to \Sigma(\Aa)$ as well as $\Rep(\Bb) \to \Rep(\Aa)$.
If $\{\pi_n\}_n \subset \Rep(\Bb)$ converges to $\pi \in \Rep(\Bb)$ in the point-norm topology (thus in the topology of $d_K$ for any compact generating set $K$) then for every $b \in \Bb$ we have $\pi_n(b) \to \pi(b)$. Thus if $\alpha: \Aa \to \Bb$ is a *-homomorphism then in particular for every $a \in \Aa$ we have $\pi_n(\alpha(a)) \to \pi(\alpha(a))$ thus $\alpha^*(\pi_n)=\pi_n \circ \alpha \to \pi \circ \alpha=\alpha^*(\pi) $ in the point-norm topology. Therefore $\alpha^*$ is always a \textit{continuous} mapping.

If $\alpha: \Aa \to \Bb,\beta: \Bb \to \Cc$ are two *-homomorphisms of $C^*$-algebras then we have $(\beta \circ \alpha)^*=\alpha^* \circ \beta^*$ thus the composition of uniformly continuous *-homomorphisms is still uniformly continuous *-homomorphism. In fact the following theorem is true:

\begin{thm} If $\alpha:\Aa \to \Bb$ is a *-homomorphism then $\alpha$ is uniformly continuous.
\end{thm}
\begin{proof}
First let us note that the above definition does not depend from the choice of compact generating sets $K$ i $L$: indeed, if $K' \subset \Aa, L' \subset \Bb$ are other compact generating sets then from Theorem ~\ref{rownowazne} the metrics $d_K$ and $d_{K'}$ as well as $d_L$ and $d_{L'}$ are pairwise  uniformly equivalent. Therefore the uniform continuity of $\alpha^*$ does not depend from the choice of metrics $d_K,d_L$. Thus let $K \subset \Aa, L \subset \Bb$ be two compact generating sets and assume that $\alpha(K) \subset L$. Therefore for $\pi,\pi' \in \Rep(\Bb)$ we have:
\begin{gather*} d_K(\alpha^{*}(\pi),\alpha^{*}(\pi'))=d_K(\pi \circ \alpha, \pi' \circ \alpha)=
\sup_{a \in K}\|\pi(\alpha(a))-\pi'(\alpha(a))\| \\
\leq \sup_{b \in L}\|\pi(b)-\pi'(b)\|=d_L(\pi,\pi')
\end{gather*}
which shows that the mapping $\alpha^*$ is uniformly continuous.
\end{proof}

Thus for an arbitrary *-homomorphism $\alpha: \Aa \to \Bb$ the mapping $$\alpha^*:\Big(\Rep(\Bb),d_L \Big) \to \Big(\Rep(\Aa),d_K \Big)$$ has its modulus of continuity which will be denoted by $\tilde{\w}_{\alpha}^{K,L}$. Then the following inequality holds:
\begin{equation} \label{modul}
d_K(\pi \circ \alpha,\pi' \circ \alpha) \leq \tilde{\w}_{\alpha}^{K,L}\big(d_L(\pi,\pi')\big), \ \ \pi,\pi' \in \Rep(\Bb).
\end{equation}
\begin{fact} Using the above notation the following equality holds: $\w_{\alpha}^{K,L}=\tilde{\w}_{\alpha}^{K,L}$.
\end{fact}
\begin{proof} For the proof of the inequality $\w_{\alpha}^{K,L} \leq \tilde{\w}_{\alpha}^{K,L}$ it suffices to show that for any $a \in \Aa$, $f^L_{\alpha(a)} \leq f^K_a \circ \tilde{\w}_{\alpha}^{K,L}$---it is thus enough to prove that for  $t \geq 0$ we have:
\begin{equation}\label{suprema}
\sup_{d_L(\pi,\pi') \leq t}\|\pi(\alpha(a))-\pi'(\alpha(a))\|   \leq \sup\{\|\tau(a)-\tau'(a)\|: d_K(\tau,\tau') \leq \tilde{\w}_{\alpha}^{K,L}(t)\}.
\end{equation}
Take $\pi,\pi' \in \Rep(\Bb)$ satisfying $d_L(\pi,\pi') \leq t$. Then putting $\tau:=\pi \circ \alpha,\tau':=\pi' \circ \alpha$ we infer from ~\eqref{modul} that:
$$d_K(\tau,\tau') \leq \tilde{\w}_{\alpha}^{K,L}\big(d_L(\pi,\pi')\big) \leq \tilde{\w}_{\alpha}^{K,L}(t).$$
Hence $\tau,\tau'$ belong to the set over which we take supremum on the right hand side of \eqref{suprema} which proves \eqref{suprema}. \par
For the proof of the opposite inequality it is enough to show that for $\pi,\pi' \in \Rep(\Bb)$  we have
\begin{equation}\label{przeciwna}
d_K(\pi \circ \alpha,\pi' \circ \alpha) \leq \w_{\alpha}^{K,L}\big(d_L(\pi,\pi')\big).
\end{equation}
If $a_0 \in K$ and $\pi,\pi' \in \Rep(\Aa)$ satisfy $d_K(\pi,\pi') \leq t$ then
$\|\pi(a_0)-\pi'(a_0)\| \leq t$ and taking supremum over all such representations we obtain $f_{a_0}^K(t) \leq t$ thus for any $t \geq 0$ we conclude $\w_{a_0}^K(t) \leq t$. Thus for any  $a \in K$ and any $\pi,\pi' \in \Rep(\Bb)$ we get 
$$\|\pi(\alpha(a))-\pi'(\alpha(a))\| \leq \w^L_{\alpha(a)}\big(d_L(\pi,\pi')\big) \leq $$
$$ \leq (\w_a^K \circ \w_{\alpha}^{K,L}) \big(d_L(\pi,\pi')\big) \leq \w^{K,L}_{\alpha}\big(d_L(\pi,\pi')\big),$$
and taking supremum over $a \in K$ we get
$$d_K(\pi \circ \alpha,\pi' \circ \alpha) \leq \w_{\alpha}^{K,L}\big(d_L(\pi,\pi')\big),$$ which proves ~\eqref{przeciwna}.
\end{proof}
The above result shows that both definitions of uniform continuity produce the same modulus of continuity for a *-homomorphism $\alpha:\Aa \to \Bb$. Thus we are allowed to write $\w_{\alpha}^{K,L}$ or shortly $\w_{\alpha}$ when it is understood which compact generating sets $K,L$ we have chosen.
\subsection{Convergence of morphisms}
On the set $\Hom(\Aa,\Bb)$ we can consider:
\begin{itemize}
\item The point-norm topology: in this topology convergence $\alpha_s \to \alpha$ holds iff $\alpha_s(a) \stackrel{\| \cdot \|}{\longrightarrow} \alpha(a)$ for every $a \in \Aa$,
\item compact-open topology: in this topology we have that $\alpha_s \to \alpha$ iff for any compact set $K \subset \Aa$ we have $$\sup_{a \in K}\|\alpha_s(a)-\alpha(a)\| \to 0.$$
\end{itemize}
We can also define a topology using a metric in a similar fashion as we did for the set $\Rep(\Aa)$, namely: we choose a compact generating set $K \subset \Aa$ and put for $\alpha,\beta \in \Hom(\Aa,\Bb)$
\begin{equation}\label{metryka}
d_K(\alpha,\beta):=\sup\{\|\alpha(a)-\beta(a)\|: a \in K\}.
\end{equation}

\begin{fact} The formula \eqref{metryka} defines a metric on the set $\Hom(\Aa,\Bb)$.
\end{fact}
\begin{proof} The proof is similar to the proof of the Fact ~\ref{metryka1}.
\end{proof}

Moreover the topology of $d_K$ on $\Hom(\Aa,\Bb)$ behaves similarly as in the case of $\Rep(\Aa)$---we have the following:

\begin{thm} The topology of $d_K$ coincides with the point-norm topology and with the compact open topology.
\end{thm}
\begin{proof} See the proof of the Theorem \ref{punktnorm}.
\end{proof}

\begin{rem} We can also define the convergence $\alpha_n \to \alpha$ using the condition $\alpha_n^* \rightrightarrows \alpha^*$. As for any $x \in \Aa$ we have $\|x\|=\sup\{\|\pi(x)\|: \pi \ \textup{-irreducible representation}\}=\sup\{\|\pi(x)\|: \pi \in \Rep(\Aa)\}$, we get
\begin{gather*} \sup_{\pi \in \Sigma(\Bb)}d_K(\pi \circ \alpha,\pi \circ \beta)=\sup\{\|\pi(\alpha(a)-\pi(\beta(a))\|: a\in K, \pi \in \Sigma(\Bb)\}= \\
=\sup_{a \in K}\|\alpha(a)-\beta(a)\|=d_K(\alpha,\beta)
\end{gather*}
thus so defined convergence coincides with the point-norm convergence (irrelevant whether we view $\alpha$ as acting between $\Sigma(-)$'s or $\Rep(-)$'s).
\end{rem}

Finally we can show that the space $(\Hom(\Aa,\Bb),d_K)$ is complete using a similar proof as for $(\Rep(\Aa),d_K)$. Let us note that the above results do not need the assumption of $\Bb$ being separable---in particular for $\Bb=\Bb(\ell^2)$ we get our previous results for representations. \\
If $\Aa,\Bb$ are commutative $C^*$-algebras then $\Aa \cong C(X)$ and $\Bb \cong C(Y)$ for some compact spaces $X,Y$. Then the set $\Hom(\Aa,\Bb)$ can be identified with $C(Y,X)$ (all continuous mappings from $Y$ to $X$). Therefore we can think of elements of $\Hom(\Aa,\Bb)$ as continuous mappings between  ,,noncommutative'' (nonexistent) compact spaces. In such circumstances we can ask about an analogue of  the Ascoli theorem, i.e. about finding necessary and sufficient conditions  for a given family $\FFf \subset \Hom(\Aa,\Bb)$ to be (pre)compact.
We will use the notation $\FFf(a):=\{\alpha(a): \alpha \in \FFf \} \subset \Bb$.

\begin{thm} Let $\Aa,\Bb$ be two unital separable $C^*$-algebras and assume that $\FFf \subset \Hom(\Aa,\Bb)$. The following conditions are equivalent:
\begin{itemize}
\item The family $\FFf$ is precompact (in the point-norm topology).
\item For every $a \in \Aa$ the family $\FFf(a) \subset \Bb$ is precompact.
\end{itemize}
\end{thm}
\begin{proof}
Since for every $a \in \Aa$ the mapping
$$\FFf \ni \alpha \mapsto \alpha(a) \in \FFf(a) \subset \Bb$$
is continuous, hence if $\FFf$ is precompact then $\FFf(a)$ also has this property. \par
Conversely, suppose that for every $a \in \Aa$ the family $\FFf(a) \subset \Bb$ is precompact. Consider the mapping
$$\iota: \cl{\FFf} \ni \alpha \mapsto \big(\alpha(a)\big)_{a \in \Aa} \in \prod_{a \in \Aa} \cl{\FFf(a)}$$
where in the codomain we consider the point-norm topology. Then $\iota$ is an embedding. Similarly as in the proof of the Theorem ~\ref{subh} we conclude that $\iota$ is closed. As we assume that for $a \in \Aa$ the set $\cl{\FFf(a)}$ is compact then from the Tychonoff's theorem the product $\prod_{a \in \Aa}\cl{\FFf(a)}$ is compact as well. Thus $\iota(\cl{\FFf})=\cl{\iota(\FFf)}$ is a closed set of the compact set thus is compact as well and this shows that $\FFf$ is precompact.
\end{proof}

\begin{cor} Suppose that $\Aa$ and $\Bb$ are two unital, separable $C^*$-algebras and let $\FFf  \subset \Hom(\Aa,\Bb)$.
\begin{itemize}
\item If $\Bb$ has Ascoli property then $\FFf$ is relatively compact if and only if for every $a \in \Aa$ the set $\FFf(a)$ is bounded, pointwise relatively compact and equicontinuous.
\item If $\Bb$ has strong Ascoli property then $\FFf$ is relatively compact if and only if for every $a \in \Aa$ the set $\FFf(a)$ is bounded and equicontinuous.
\end{itemize}
\end{cor}

On the set $\Hom(\Aa,\Bb)$ we can also define topology using the following notion of convergence: $\alpha_s \to \alpha$ if for any $\pi \in \Sigma(\Bb)$ we have $\pi \circ \alpha_s \to \pi \circ \alpha$ in the point-norm topology. In other words this is the pointwise convergence  $\alpha_s^* \to \alpha^*$ where $\alpha_s^*,\alpha^*:\Sigma(\Aa) \to \Rep(\Aa)$. If $\alpha_n \to \alpha$ in the point-norm topology then for every $a \in \Aa$ we get $\alpha_n(a) \to \alpha(a)$ and for every $\pi \in \Sigma(\Bb)$ we have $\pi(\alpha_n(a)) \to \pi(\alpha(a))$. The converse implication need not hold:
\begin{exs}
\begin{enumerate}
\item Let $\Aa_1=C(X),\  \Aa_2=C(Y)$ be two commutative $C^*$-algebras and $\alpha_n,\alpha:C(Y) \to C(X)$ be *-homomorphisms. Then there exist continuous mappings $u_n,u:X \to Y$ such that $\alpha_n(f)=f \circ u_n, \alpha(f)=f \circ u$. Then the point-norm convergence $\alpha^*_n \to \alpha^*$
is equivalent to the convergence $\|f \circ u_n-f \circ u\|_{\oo} \to 0$ (and this condition is equivalent to the uniform convergence $u_n \rightrightarrows u$). On the other hand, the convergence defined above is equivalent to the pointwise convergence $f \circ u_n \to f \circ u$  (which is equivalent to the pointwise convergence $u_n \to u$).
\item Let $\Aa=\Kk(\ell^2)^+$. Then the identity representation is irreducible and thus the point-norm convergence is equivalent to the above defined convergence.
\end{enumerate}
\end{exs}
Finally we can consider also the uniform convergence $\alpha^*_n \to \alpha^*$---however this yields nothing new, since it is then equivalent to the point-norm convergence (or convergence in the metric $d_K$). Indeed we have
$$\sup_{\pi \in \Sigma(\Bb)}d_K(\alpha^*_n(\pi),\alpha^*(\pi))=\sup_{\pi \in \Sigma(\Bb)}\sup_{a \in K} \|\pi(\alpha_n(a))-\pi(\alpha(a))\|=$$
$$=\sup_{a \in K}\|\alpha_n(a)-\alpha(a)\|=d_K(\alpha_n,\alpha).$$

\section{Comparison of moduli of continuity with respect to $\Sigma(-)$ and $\Rep(-)$}
Let $(X,d)$ be a metric space with diameter $R>0$ and let
$$K:=\{f:X \to [-\frac{R}{2},\frac{R}{2}]: \Lip(f) \leq 1\}$$
where $\Lip(f)$ denotes the Lipschitz constant for $f$. Equivalently one can put $K:=\{f:X \to \RRR: \Lip(f) \leq 1, f(a)=0 \ \textup{for some}\  a \}$.
Then $d(x,y)=\sup_{f \in K}|f(x)-f(y)|$. For $f:X \to \RRR$ we have that $f$ is a contraction, i.e. $\Lip(f) \leq 1$ if and only if there exists $c \in \RRR$ such that $f-c \in K$. Thus for a real valued function $f$ the condition $\Lip(f) \leq s$ is equivalent to the existence of  $c \in \RRR$ such that $\frac{1}{s}f-c \in K$. \\
Now fix two representations $\pi,\pi':C(X) \to \Bb(\ell^2)$ and take $f \in C(X;\RRR)$.
Therefore the (classical) modulus of continuity $\w:=\w_f$ for $f$ coincides with the modulus of continuity $\w_f^{\Sigma}$ for $f$ viewed as an element in $C^*$-algebra, constructed using irreducible representations and we have $\w_f=\w_f^{\Sigma} \leq \w_f^{\Rep}$ where $\w_f^{\Rep}$ is the modulus of continuity for $f$ constructed using all representations in $\Rep(C(X))$. We will show that indeed an equality holds for a real valued function. In order to do so we will need the so called \textit{Fenchel transform} (also called \textit{convex conjugate}). The context is the following: we consider a finite dimensional euclidean space $E$ with the inner product $\langle \cdot,- \rangle$ and functions $h:E \to [-\oo,\oo]$. We define the \textit{Fenchel conjugate} of $h$ by the formula
$$h^*(\varphi)=\sup_{x \in E}(\langle \varphi,x \rangle-h(x)), \quad h^*:E \to [-\oo,\oo].$$
If $h$ is not identically equal to $+\oo$ then $h^*(\varphi)>-\oo$ for every $\varphi \in E$. By the direct calculation one can check that $h^*$ is always a convex function and also that the inequality $h^{**} \leq h$ always holds. It is natural to ask when $h=h^{**}$ or in other words when $h=g^*$ for some function $g$. The answer is provided by the following theorem (see Thm. 4.2.1. in \cite{Bor}):
\begin{thm} [Fenchel-Moreau] Let $h:E \to (-\oo,\oo]$ be any function. The following conditions are equivalent:
\begin{enumerate}
\item $h^{**}=h$,
\item $h=g^*$ for some function $g:E \to (-\oo,\oo]$,
\item $h$ is convex and the set $\{(x,t) \in E \times (-\oo,\oo]:h(x) \leq t\}$ is closed,
\item for any $x \in E$, $h(x)=\sup\{\alpha(x): \alpha \leq h, \alpha-\textup{affine function}\}$.
\end{enumerate}
\end{thm}

Put for $s \geq 0$
$$\delta(s)=\frac{1}{2}\sup_{t \geq 0}(\w_f(t)-st).$$
\begin{rem} One can show that $\delta(s)$ is the distance of $f$ to the set of all Lipschitz function with Lipschitz constant $\leq s$---namely that the formula:
$\delta(s)=\inf\{\|f-u\|_{\oo}: \Lip(u) \leq s\}<\oo$ is satisfied.
\end{rem}
Then, using the Fenchel transform, one can obtain the relation between $\delta$ and the (classical) modulus of continuity for $f$, namely:
\begin{thm}
Using the above notation and assumption one has
$$\w_f(t)=\inf_{s \geq 0} (2\delta(s)+st).$$
\end{thm}
\begin{proof}
Denote $\w:=\w_f$ and put $\alpha(t)=
\begin{cases}
-\w(-t) \quad if \quad t \leq 0, \\
+\oo \quad if \quad t>0
\end{cases}$,
and also extend the function $\delta$ by putting $\delta(s)=+\oo$ for $s<0$.
Then $\alpha$ is a convex function. Moreover, for $s \leq 0$ we have:
\begin{gather*}
\alpha^*(s)=\sup_{t \in \RRR}\big(st-\alpha(t)\big)=\sup_{t \leq 0}\big(st-\alpha(t)\big)= \\
=\sup_{t \leq 0}\big(st+\w(-t)\big)=\sup_{t \geq 0}\big(\w(t)-st\big)=2\delta(s)
\end{gather*}
(for $s>0$ also holds $\alpha^*(s)=2\delta(s)$ since both sides are infinite).
Hence $\alpha$ satisfies the assumptions of the Fenchel-Moreau Theorem: indeed, $\alpha$ takes finite values on a closed interval and is continuous there, thus the set $\{(t,u) \in \RRR \times (-\oo,\oo]:\alpha(t) \leq u\}$ is closed. Hence from the Fenchel-Moreau Theorem $\alpha^{**}=\alpha$.
Therefore
$$\alpha(s)=\big(2\delta(s)\big)^*=\sup_{t \in \RRR}\big(st-2\delta(t)\big)=
\sup_{t \geq 0}\big(st-2\delta(t)\big),$$
hence for $s \leq 0$ we get
$$\w(-s)=-\sup_{t \geq 0}\big(st-2\delta(t)\big)=\inf_{t \geq 0}\big(2\delta(t)-st\big).$$
Substituting $-s$ for $s \geq 0$ we get the desired result.
\end{proof}
The function realising the distance in the definition of $\delta$ is:
$$f_s:=\delta(s)+\inf_{y \in X}(f(y)+sd(\cdot,y)).$$
Then $\Lip(f_s) \leq s$ (indeed $\Lip(sd(\cdot,y)) \leq s$ and after taking translations this Lipschitz constant remains the same so is after taking infimum), therefore there exists $c \in \RRR$ such that $\frac{1}{s}f_s-c \in K$. Moreover we obtain $\|f-f_s\|_{\oo}=\delta(s)$, since:
\begin{gather*}
f(x)-f_s(x) \leq \delta(s) \iff
f(x)-\delta(s)-\inf_{y \in X}\{f(y)+sd(x,y)\} \leq \delta(s) \iff \\
f(x)+\sup_{y \in X}\{-f(y)-sd(x,y)\} \leq 2 \delta(s) \iff
\sup_{y \in X} \{f(x)-f(y)-sd(x,y)\} \leq 2\delta(s) \\
\end{gather*}
but the last inequality is satisfied since
$$f(x)-f(y)-sd(x,y) \leq |f(x)-f(y)|-sd(x,y) \leq \w\big(d(x,y)\big)-sd(x,y) \leq 2\delta(s)$$
and it suffices to take supremum. \\
Conversely:
\begin{gather*}
f_s(x)-f(x) \leq \delta(s) \iff \delta(s)+\inf_{y \in X}\{f(y)+sd(x,y)\}-f(x) \leq \delta(s) \iff \\
\inf_{y \in X}\{f(y)-f(x)+sd(x,y)\} \leq 0
\end{gather*}
but for $y=x$ we get $f(y)-f(x)+sd(x,y)=0$ thus the infimum of this expression is $\leq 0$. We have shown that $|f(x)-f_s(x)| \leq \delta(s)$ hence $\|f-f_s\| \leq \delta(s)$. \par
Now let $\pi,\pi' \in \Rep(C(X))$---in this case we estimate:
$$\|\pi(f_s)-\pi'(f_s)\|=s\|\pi\big(\frac{1}{s}f_s-c\big)-\pi'\big(\frac{1}{s}f_s-c\big)\| \leq sd_K(\pi,\pi')$$ hence we get:
\begin{gather*}
\|\pi(f)-\pi'(f)\| \leq \|\pi(f)-\pi(f_s)\|+\|\pi(f_s)-\pi'(f_s)\|+\|\pi'(f_s)-\pi'(f)\| \\
 \leq \|f-f_s\|+sd_K(\pi,\pi')+\|f-f_s\| \leq 2\delta(s)+sd_K(\pi,\pi').
\end{gather*}
Fixing $t \geq 0$ and taking supremum gives:
$$\sup\{\|\pi(f)-\pi'(f)\|: \pi,\pi' \in \Rep(C(X)), d_K(\pi,\pi') \leq t \} \leq \inf_{t \geq 0}(2\delta(s)+st)=\w(t)$$
which yields $\w_f^{\Rep}=\w$. Therefore we see that for a \textit{real valued} function 
it does not matter whether we consider its modulus of continuity defined using \textit{all} representations or using only \textit{irreducible} ones. \par
In the general (complex-valued) case we have the following situation: for $f \in C(X)$ express $f$ as $f=u+iv$ where $u,v$ are real-valued. We obtain
$$\w_f=\w_{u+iv} \leq \w_u+\w_v,$$
$$\w_u=\w_{\frac{f+f^*}{2}} \leq \underbrace{\w_{\frac{f}{2}}}_{=\frac{1}{2}\w_f}+\underbrace{\w_{\frac{f^*}{2}}}_{=\frac{1}{2}\w_f}=\w_f,$$
Similarly we have:
$$\w_v=\w_{\frac{f-f^*}{2i}} \leq \w_{\frac{f}{2}}+\w_{\frac{f^*}{2}} \leq \w_f$$
from which it follows:
\begin{equation}\label{rownowazne1}
\w_f \leq \w_f^{\Rep} \leq \w_u^{\Rep}+\w_v^{\Rep}=\w_u+\w_v \leq 2\w_f.
\end{equation}

\section{Interpretation of $\Rep(-)$}
We have used the space $\Rep(\Aa)$ heavily in our considerations, therefore let us briefly discuss the interpretation of this space in the commutative case where $\Aa=C(X)$. It turns out that $\Rep(\Aa)$ has the following interpretation:
$$\Rep\big(C(X)\big)=\{\textup{spectral\ measures\ on $X$}\}.$$
We briefly explain how do we understand the above identification:
let $\pi:C(X) \to \Bb(\HHh)$ be a *-representation and let $\xi,\eta \in \HHh$.
Consider the mapping
$$\varphi_{\xi,\eta}:C(X) \ni f \to \langle \pi(f)\xi,\eta \rangle \in \CCC.$$
This is a linear and bounded functional: it satisfies $\|\varphi_{\xi,\eta}\| \leq \|\xi\|\|\eta\|$, thus from the Riesz representation theorem there exists exactly one (regular, Borel) measure $\mu_{\xi,\eta}$ such that
$$\langle \pi(f)\xi,\eta \rangle=\int_X fd\mu_{\xi,\eta}$$
and $\| \varphi_{\xi,\eta} \|=\| \mu_{\xi,\eta} \|(=|\mu_{\xi,\eta}|(X))\leq \|\xi\|\|\eta\|$. Fix a Borel set $A$ and consider the mapping
$$(\xi,\eta) \mapsto \mu_{\xi,\eta}(A).$$
Then it is a sesquilinear, bounded form on $\HHh$, thus it corresponds to the unique (bounded) operator which we denote by $E(A)$. By the direct calculation we check that the following conditions are satisfied:
\begin{itemize}
\item $0 \leq E(A) \leq I$,
\item $E(X)=I$,
\item $E(\bigcup_{n=1}^{\oo}A_n)=WOT-\sum_{n=1}^{\oo}E(A_n)$ for a sequence of pairwise disjoint Borel sets $\{A_n\}_{n \in \NNN}.$
\end{itemize}
A property which is nontrivial is $E(A \cap B)=E(A)E(B)$. This equality is true for all sets which are measurable with respect to the sigma algebra generated by closed $G_{\delta}$ sets.
In the metrisable case every closed set is automatically $G_{\delta}$ thus the equality holds for all sets which are measurable with respect to sigma algebra generated by all closed sets, i.e. for all Borel sets. In the general, not necessarily metrisable case, one can use the regularity of the measures $\mu_{\xi,\eta}$ together with the fact that for any pair of sets  $K,U$ such that $K$ is closed, $U$ is open and $K \subset U$ one can find a closed $G_{\delta}$ set $F$ such that $K \subset F \subset U$.

In this context it is worth to mention the formula which is due to Kantorovich and which allows to extend the metric from the space $X$ to the space of all possible probabilistic measures on $X$. In more details: if $(X,d)$ is a metric space then one can extend the metric $d$ to the set $\Prob(X)$ of all regular, Borel, probabilistic measures on $X$ using the formula
\begin{equation}\label{Kantorovich}
\tilde{d}(\mu,\nu)=\sup\{\Big|\int_Xfd\mu-\int_Xfd\nu\Big|: f:X \to \RRR,\ f \ \textup{is a contractive map }\}.
\end{equation}
We are using here the identification $x \simeq \delta_x$ where
$\delta_x(A)=
\begin{cases}
1, \ x \in A \\
0, \ x \notin A.
\end{cases}$ \\
Then
$$\Big|\int_Xfd\delta_x-\int_Xfd\delta_y\Big|=|f(x)-f(y)|$$
hence $\tilde{d}(\delta_x,\delta_y)=d(x,y)$ thus indeed we get an extension of $d$.
It is not hard to show that $\tilde{d}$ indeed defines a metric. Moreover one can describe the topology induced by $\tilde{d}$:
\begin{thm} The topology induced by the metric $\tilde{d}$ coincides with the weak-* topology.
\end{thm}
\begin{proof}
The proof is analogous to the proof of Theorem \ref{punktnorm}.
\end{proof}
What is more, the mapping $(X,d) \mapsto (\Prob(X),\tilde{d})$ is functorial: any continuous map $f:(X,d_X) \to (Y,d_Y)$ induces $f_*:(\Prob(X),\tilde{d_X}) \to (\Prob(Y),\tilde{d_Y})$ via the transport of measure $f_*(\mu)=\mu \circ f^{-1}$. Those \textit{pushforwards} satisfy $(f\circ g)_*=f_* \circ g_*$. If $(X,d)$ is an infinite metric space then $(\Prob(X),\tilde{d})$ is homeomorphic to the Hilbert cube (this follows from Keller's Theorem, see e.g. \cite{Bes}, Theorem 3.1).

\section{Further problems}
Finally let us try to discuss briefly the questions which are natural and for which we did not give the answer. \par
We have investigated various properties of compact $C^*$-algebras: in particular we have seen that the class of compact $C^*$-algebras is closed under taking ideals, direct sums as well as quotient operation. It is natural to ask the following: 
\begin{prob}
Is it true that if $\Aa$ is a compact $C^*$-algebra (unital and separable) and $\Bb \subset \Aa$ is a $C^*$-subalgebra then $\Bb$ is compact as well?
\end{prob}

Assume that $K$ is a compact set generating $C^*$-algebra $\Aa$. Then for $a \in \Aa$ we can consider  the moduli of continuity $\w^K_a$ and define the set $\tilde{K}=\{a \in \Aa: \w^K_a(t) \leq t,\ t \in [0,\oo)\}$. Then obviously $K \subset \tilde{K}$, so in particular $\tilde{K}$ also generates $\Aa$. Obviously $\tilde{K}$ is not bounded since for $\lambda \in \CCC$ we have $\w^K_{\lambda 1}=0$. But still we can define $K':=\tilde{K} \cap B$ where $B$ is the (closed) ball with radius $\diam K$. In this way we obtain a bounded generating set---hence it makes sense to consider the metric $d_{K'}$.
\begin{prob} What is the topology of $d_{K'}$?
\end{prob}

\begin{rem}
Note that obviously  $K'$ is equicontinuous and bounded---thus if $\Aa$ has strong Ascoli property then $K'$ is compact and $d_{K'}$ is uniformly equivalent to $d_K$. In particular this will hold for compact $C^*$-algebras.
\end{rem}
In our considerations we have defined the notion of pointwise convergence using the family of seminorms $a \mapsto \|\pi(a)\|$ where $\pi \in \Sigma(\Aa)$. We can also consider only those seminorms which come from irreducible representations (or more precisely representations of the form $\aleph_0 \odot \pi$ where $\pi$ is irreducible). This is obviously a weaker condition. The natural question is whether these two notions coincide:
\begin{prob} Is it possible to construct a net $(a_s)_s \subset \Aa$ such that for any irreducible representation $\pi \in \Irr(\Aa)$ we get $\pi(a_s) \to 0$ but $(a_s)_s$ do not converge to $0$ pointwise?
\end{prob}

Finally it would be interesting to know how the set $\Sigma(\Aa)$ may look like for various $C^*$-algebras. For example we have already seen that if $\Aa$ is $N$-subhomogenous $C^*$-algebra then
$\Sigma(\Aa)$ contains at most $N$-dimensional representations (more precisely, representations of the form  $\aleph_0 \odot \pi$ where $\pi$ is at most $N$-dimensional).

\begin{prob} Is it true that if $\Aa$ is a unital CCR-algebra then it may happen that $\Sigma(\Aa)$ contains a representation which is not of the form $\aleph_0 \odot \pi$ where $\pi$ is finite dimensional?
\end{prob}


\addcontentsline{toc}{section}{Bibliografia}

\begin{thebibliography}{20}


\bibitem{Aro} N. Aronszajn, P. Panitchpakdi \textit{Extension of uniformly continuous transformations and hyperconvex metric spaces}, Pacific J. Math. 6 (1956), no. 3, 405--439.

\bibitem{Arv} W. Arveson, \textit{An invitation to $C^*$-algebras}, Springer Verlag, New York, 1976.

\bibitem{Bes} C. Bessaga, A. Pe³czyñski, \textit{Selected topics in infinite-dimensional topology}, PWN, Warszawa, 1975.

\bibitem{Bla} B. Blackadar, \textit{Operator Algebras. Theory of $C^*$-algebras and von Neumann Algebras}, Springer-Verlag, Berlin, 2006.

\bibitem{Bor} J.M. Borwein, A.S. Lewis, \textit{Minicourse Notes based on Convex Analysis and Nonlinear Optimalization}, Springer Verlag, 1999.

\bibitem{Oza} N.P. Brown, N. Ozawa, \textit{$C^*$-Algebras and Finite-Dimensional Approximations}, American Mathematical Society, Providence, Rhode Island, 2008.

\bibitem{Cob} L. Coburn, \textit{The $C^*$-algebra generated by an isometry}, Bull. Amer. Math. Soc. Volume 73, Number 5, 1967, 722-726.

\bibitem{Con} A. Connes, \textit{Noncommutative Geometry}, Academic Press, London, 1994.

\bibitem{Dav} K.R. Davidson, \textit{$C^*$-algebras by example}, American Mathematical Society, Providence, 1991.

\bibitem{Dix} J. Dixmier, \textit{$C^*$-algebras}, North-Holland Publishing Company, Amsterdam, 1977.

\bibitem{Fel} J.M.G. Fell, \textit{The structure of algebras of operators fields}, Acta Math. 106, 1961, 233-280.

\bibitem{Var} J.M. Gracia-Bondia, H. Figueroa, J. Varilly, \textit{Elements of noncommutative geometry}, Birkhauser Advanced Texts, Boston, Berlin, 2001.

\bibitem{Hew2} E. Hewitt, K. Ross, \textit{Abstract Harmonic Analysis volume 2}, Springer Verlag, New York, Berlin 1997.

\bibitem{KR1} R.V. Kadison, J.R. Ringrose, \textit{Fundamentals of the theory of the operator Algebras, volume 1}, Academic Press, New York, 1983.

\bibitem{KR2} R.V. Kadison, J.R. Ringrose, \textit{Fundamentals of the theory of the operator Algebras, volume 2}, Academic Press, London, 1986.

\bibitem{Kha} M. Khalkhali, \textit{Basic Non-commutative Geometry}, European Mathematical Society, 2009.

\bibitem{Kur} K. Kuratowski, \textit{Topology vol. 1}, PWN, Warszawa, 1966.

\bibitem{Liu} Z. Liu, \textit{On some mathematical aspects of the Heisenberg relation},
Sci. China 54, (2011), 2427--2452.

\bibitem{Nie1} P. Niemiec, \textit{Unitary equivalence and decompositions of finite systems of closed densely defined operators in Hilbert spaces}, Dissertationes Math. (Rozprawy Mat.) 482, 2012.

\bibitem{Nie2} P. Niemiec, \textit{Elementary approach to homogeneous $C^*$-algebras}, Rocky Mountain J. Math. 45 (2015), 1591-1630.

\bibitem{Nie4} P. Niemiec, \textit{Models for subhomogeneous $C^*$-algebras}, arXiv: 1310.5595.

\bibitem{Ped} G.K. Pedersen, \textit{$C^*$-algebras and their automorphism groups}, Academic Press, London, 1979.

\bibitem{Phe} R.R. Phelps, \textit{Lecture notes on Chouquet's theorem}, Springer Verlag, Berlin, Heidelberg, 2001.

\bibitem{Rae} I. Raeburn, D. Williams, \textit{Morita equivalence and continuous trace $C^*$-algebras}, Mathematical Surveys and Monographs, Vol. 60, Providence, 1998.

\bibitem{Rie} M. Rieffel, \textit{Compact quantum metric spaces}, Operator algebras, quantization and noncommutative geometry, Contemp. Math. 365, Amer. Math. Soc., Providence, 2004, 315-330.

\bibitem{Rud} W. Rudin \textit{Functional Analysis}, PWN, Warszawa, 2011.

\bibitem{Sak} S. Sakai, \textit{$C^*$-algebras and $W^*$-algebras}, Springer-Verlag, Berlin, 1971.

\bibitem{Sch} J.T. Schwartz, \textit{$W^*$-algebras}, Gordon and Breach, Science Publishers Inc., New York, 1967.

\bibitem{She} D. Sherman, \textit{Unitary orbits of normal operators in von Neumann algebras},
Journal f{\"u}r die reine und angewandte Mathematik 605, 2007, 95-132.

\bibitem{Tak} M. Takesaki, \textit{Theory of Operator Algebras I}, Springer-Verlag, Berlin, 2002.

\bibitem{Vas} N.B. Vasil'ev, \textit{$C^*$-algebras with finite dimensional irreducible representations}, Russian Math. Surveys 21, 1966, 137-155.

\bibitem{Weg} N.E. Wegge-Olsen, \textit{$K$-theory and $C^*$-algebras. A friendly approach},
Oxford University Press, Oxford, 1993.

\end{thebibliography}
\end{document}